\documentclass[12pt,reqno]{amsart}
\usepackage{latexsym}
\usepackage{amssymb}
\usepackage{mathrsfs}
\usepackage{amsmath}
\usepackage{fancybox,color}
\usepackage{enumerate}
\usepackage{enumitem}
\usepackage[latin1]{inputenc}
\usepackage{soul} 
\usepackage[colorlinks=true, linkcolor=magenta, citecolor=magenta]{hyperref}

\usepackage{color}

\newcommand{\ed}{\color{black}}

\newcommand{\cprime}{$'$}

\usepackage{comment}

\numberwithin{equation}{section}

\makeatletter %modifies the appearance of subsection headings
\renewcommand\subsection{\@startsection{subsection}{2}%
  \z@{-.5\linespacing\@plus-.7\linespacing}{.3\linespacing}%
  {\normalfont\bfseries}}
\makeatother

\usepackage[left=1.92cm,top=1.83cm,right=1.92cm]{geometry}

\def\1{\raisebox{2pt}{\rm{$\chi$}}}

\newtheorem{theorem}[equation]{Theorem}

\newtheorem{lemma}[equation]{Lemma}

\theoremstyle{definition}
\newtheorem{definition}[equation]{Definition}
\newtheorem{remark}[equation]{Remark}

\newtheorem{example}[equation]{Example}

\newcommand{\R}{{\mathbb R}}

\newcommand{\N}{{\mathbb N}}

\newcommand\diam{\operatorname{diam}}

\newcommand{\Ha}{{\mathcal H}}

\DeclareMathOperator{\ucodima}{\overline{co\,dim}_A}

\DeclareMathOperator{\cp}{cap}

{\catcode`p =12 \catcode`t =12 \gdef\eeaa#1pt{#1}}      % Get slantfactor
\def\accentadjtext#1{\setbox0\hbox{$#1$}\kern   % Convert it with height
                \expandafter\eeaa\the\fontdimen1\textfont1 \ht0 }
\def\accentadjscript#1{\setbox0\hbox{$#1$}\kern % Convert it with height
                \expandafter\eeaa\the\fontdimen1\scriptfont1 \ht0 }
\def\accentadjscriptscript#1{\setbox0\hbox{$#1$}\kern   % Convert it with height
                \expandafter\eeaa\the\fontdimen1\scriptscriptfont1 \ht0 }
\def\accentadjtextback#1{\setbox0\hbox{$#1$}\kern       % Convert it with height
                -\expandafter\eeaa\the\fontdimen1\textfont1 \ht0 }
\def\accentadjscriptback#1{\setbox0\hbox{$#1$}\kern     % Convert it with height
                -\expandafter\eeaa\the\fontdimen1\scriptfont1 \ht0 }
\def\accentadjscriptscriptback#1{\setbox0\hbox{$#1$}\kern % Convert it with height
                -\expandafter\eeaa\the\fontdimen1\scriptscriptfont1 \ht0 }
\def\itoverline#1{{\mathsurround0pt\mathchoice
        {\rlap{$\accentadjtext{\displaystyle #1}
                \accentadjtext{\vrule height1.593pt}
                \overline{\phantom{\displaystyle #1}
                \accentadjtextback{\displaystyle #1}}$}{#1}}
        {\rlap{$\accentadjtext{\textstyle #1}
                \accentadjtext{\vrule height1.593pt}
                \overline{\phantom{\textstyle #1}
                \accentadjtextback{\textstyle #1}}$}{#1}}
        {\rlap{$\accentadjscript{\scriptstyle #1}
                \accentadjscript{\vrule height1.593pt}
                \overline{\phantom{\scriptstyle #1}
                \accentadjscriptback{\scriptstyle #1}}$}{#1}}
        {\rlap{$\accentadjscriptscript{\scriptscriptstyle #1}
                \accentadjscriptscript{\vrule height1.593pt}
                \overline{\phantom{\scriptscriptstyle #1}
                \accentadjscriptscriptback{\scriptscriptstyle #1}}$}{#1}}}}

\newcommand{\iol}{\itoverline}

\newcommand{\eps}{{\varepsilon}}
\def\1{\raisebox{2pt}{\rm{$\chi$}}}

\newcommand{\Lip}{\operatorname{Lip}}

\def\vint_#1{\mathchoice%
        {\mathop{\kern 0.2em\vrule width 0.6em height 0.69678ex depth -0.58065ex
                \kern -0.8em \intop}\nolimits_{\kern -0.4em#1}}%
        {\mathop{\kern 0.1em\vrule width 0.5em height 0.69678ex depth -0.60387ex
                \kern -0.6em \intop}\nolimits_{#1}}%
        {\mathop{\kern 0.1em\vrule width 0.5em height 0.69678ex
            depth -0.60387ex
                \kern -0.6em \intop}\nolimits_{#1}}%
        {\mathop{\kern 0.1em\vrule width 0.5em height 0.69678ex depth -0.60387ex
                \kern -0.6em \intop}\nolimits_{#1}}}
\def\vintslides_#1{\mathchoice%
        {\mathop{\kern 0.1em\vrule width 0.5em height 0.697ex depth -0.581ex
                \kern -0.6em \intop}\nolimits_{\kern -0.4em#1}}%
        {\mathop{\kern 0.1em\vrule width 0.3em height 0.697ex depth -0.604ex
                \kern -0.4em \intop}\nolimits_{#1}}%
        {\mathop{\kern 0.1em\vrule width 0.3em height 0.697ex depth -0.604ex
                \kern -0.4em \intop}\nolimits_{#1}}%
        {\mathop{\kern 0.1em\vrule width 0.3em height 0.697ex depth -0.604ex
                \kern -0.4em \intop}\nolimits_{#1}}}

\newcommand{\intav}{\vint}
\newcommand{\aveint}[2]{\mathchoice%
        {\mathop{\kern 0.2em\vrule width 0.6em height 0.69678ex depth -0.58065ex
                \kern -0.8em \intop}\nolimits_{\kern -0.45em#1}^{#2}}%
        {\mathop{\kern 0.1em\vrule width 0.5em height 0.69678ex depth -0.60387ex
                \kern -0.6em \intop}\nolimits_{#1}^{#2}}%
        {\mathop{\kern 0.1em\vrule width 0.5em height 0.69678ex depth -0.60387ex
                \kern -0.6em \intop}\nolimits_{#1}^{#2}}%
        {\mathop{\kern 0.1em\vrule width 0.5em height 0.69678ex depth -0.60387ex
                \kern -0.6em \intop}\nolimits_{#1}^{#2}}}

\title{The Haj{\l}asz capacity density condition is self-improving} 
\author[J.\! Canto]{Javier Canto}
\address[J.C.]{BCAM -- Basque Center for Applied Mathematics, Alameda de Mazarredo 14, 48009 Bilbao, Spain}
\email{javier.canto@ehu.eus}

\author[A.V.\! V\"ah\"akangas]{Antti V. V\"ah\"akangas}
\address[A.V.V.]{University of Jyvaskyla, Department of Mathematics and Statistics, P.O. Box 35, FI-40014 University of Jyvaskyla, Finland} 
\email{antti.vahakangas@iki.fi}

\makeatletter
\@namedef{subjclassname@2020}{{\mdseries 2020} Mathematics Subject Classification}
\makeatother

\keywords{Analysis on metric spaces,  capacity density condition, Haj{\l}asz gradient}
\subjclass[2020]{35A23,  31E05, 30L99, 42B25, 46E35}
\thanks{J.C.\! is supported by supported by the Ministerio de Econom\'ia y Competitividad (Spain) through grants PID2020-113156GB-I00 and SEV-2017-0718, and by Basque Government through grant IT-641-13 and BERC 2018-2021 and ``Ayuda para la formaci\'on de personal investigador no doctor".}

\begin{document}

\begin{abstract}
We prove a self-improvement  property of a  capacity density condition
for a  nonlocal Haj{\l}asz gradient in complete geodesic spaces. 
%More specifically, we define the  capacity density condition and show that it is self-improving in the setting
%of complete geodesic spaces. 
The proof relates the
capacity density condition with boundary Poincar\'e inequalities, adapts Keith--Zhong techniques
for establishing local Hardy inequalities and  applies Koskela--Zhong arguments
for proving self-improvement properties of local Hardy inequalities.
This leads to a characterization of
the Haj{\l}asz capacity density condition in terms of a strict upper
bound on the upper  Assouad codimension of the underlying set, which shows the self-improvement property of the Haj{\l}asz capacity density condition. 
\end{abstract}

\maketitle

\section{Introduction}

We introduce a Haj{\l}asz $(\beta,p)$-capacity density condition  in terms of Haj{\l}asz gradients of order $0<\beta\le 1$,  see Sections \ref{s.abstract} and \ref{s.capacity}. Our main result, Theorem \ref{t.main-improvement}, states that this condition is doubly open-ended, that is,  a Haj{\l}asz $(\beta,p)$-capacity density  condition  is self-improving both in $p$ and in $\beta$  if $X$ is a complete geodesic space.  
The study of such conditions can be traced back to the seminal work by Lewis \cite{MR946438}, who established self-improvement of  Riesz $(\beta,p)$-capacity density conditions in $\R^n$.  His result  has  been followed by other works incorporating different techniques often in metric spaces, like nonlinear potential theory \cite{MR1386213,MR1869615},  and local Hardy inequalities \cite{MR3673660}. 

A distinctive feature of our paper is that we prove the self-improvement  of a  capacity density condition
for a  nonlocal  gradient for the first time in metric spaces.
We make use of a recent advance \cite{MR3895752} in Poincar\'e inequalities, whose self-improvement properties were originally shown by Keith--Zhong in their celebrated work \cite{MR2415381}.
In this respect, we join the line of research 
initiated in \cite{KorteEtAl2011}, and continued in \cite{MR3976590, MR4116806}, for bringing together the seemingly  distinct self-improvement properties of capacity density conditions and Poincar\'e inequalities.

We use various techniques and concepts in the proof of Theorem \ref{t.main-improvement}.  
The fundamental idea is to use a geometric concept, more precisely the upper Assouad codimension, and characterize the capacity density with a  strict upper  bound on this codimension. 
 Here we are motivated by the recent approach from \cite{DLV2021}, where the Assouad codimension bound is
used to give necessary and sufficient conditions for certain fractional Hardy inequalities; we also refer to \cite{MR3631460}.
 The principal difficulty is to prove a strict bound on the codimension. To this 
end we relate  the  capacity density condition to boundary Poincar\'e inequalities, and we show their self-improvement roughly speaking in two steps: (1) Keith--Zhong estimates on  maximal functions and (2) Koskela--Zhong estimates on  Hardy inequalities.  For these purposes, respectively,  
we adapt the maximal function methods from
\cite{MR3895752} and the local Hardy arguments from \cite{MR3673660}.

One of the main challenges our method is able to overcome is the nonlocal nature of Haj{\l}asz gradients \cite{MR1876253}.  More specifically, if a function $u$ is
constant in a set $A\subset X$ and $g$ is a Haj{\l}asz gradient of $u$, then $g\mathbf{1}_{X\setminus A}$ is not necessarily
a Haj{\l}asz gradient of $u$.  This fact makes it  impossible  to directly use  the standard localization techniques. More specifically, there is no access to  neither  pointwise glueing lemma nor pointwise Leibniz rule, both of which are  used while proving similar self-improvement properties
 for capacity density conditions involving $p$-weak upper gradients,  for example, by the Wannebo approach \cite{MR1010807} that was used in \cite{MR3673660}  to show corresponding local Hardy inequalities. The Haj{\l}asz gradients satisfy nonlocal versions of these basic tools, both of which we employ in our method.
 
There is a clear advantage to working with Haj{\l}asz gradients:  Poincar\'e inequalities hold for all measures, see Section \ref{s.abstract}. Other types of gradients, 
such  as $p$-weak upper gradients  \cite{MR2867756}, do not have this property and therefore corresponding Poincar\'e inequalities need to be assumed a priori, as was the case in previous works such as \cite{MR1386213, MR1869615,MR3673660,MR3976590,MR4116806}. We remark that this  requirement  already excludes many doubling measures
 in $\R$ equipped with Euclidean distance~ \cite{BjornBuckleyKeith2006}.

Our method has also  a disadvantage. We need to assume that 
$X$ is a complete geodesic space.  These assumptions provide us Lemma \ref{l.measure_to_hausdorff},
Theorem \ref{t.qp_poincare}, Lemma \ref{l.ball_measures},  Lemma \ref{l.continuous}, and few other useful properties. 
We do not know how far these two conditions could be relaxed.
In particular, it would be interesting to know if our main result, Theorem \ref{t.main-improvement},  could be
extended to the more general setting of complete and connected metric spaces.

The outline of this paper is as follows. After a brief discussion on notation and preliminary concepts in Section \ref{s.preliminaries}, Haj{\l}asz gradients are introduced in Section \ref{s.abstract} along with their calculus and various Poincar\'e inequalities. Capacity density condition is discussed in Section \ref{s.capacity}, and some preliminary sufficient and necessary bounds on the Assouad codimension are given in Section~\ref{s.geometry}. The most technical part of the work is contained in Sections \ref{s.truncation},  \ref{s.main} and \ref{s.local_hardy}, in which the analytic framework of the self-improvement is gradually developed. Finally, the main result is given in Section \ref{s.big-results},
 in which we show that various geometrical and analytical conditions are equivalent to the capacity density condition.  The geometrical conditions are open-ended by definition, and hence all analytical conditions
are seen to be self-improving or doubly open-ended.

\section{Preliminaries}\label{s.preliminaries}

%\subsection{Tracking constants}\label{s.constants}
%
%
%Our results are based on 
%quantitative estimates and absorption arguments, where
%it is often crucial to track the dependencies of constants quantitatively.
%For this purpose, we will use the following notational convention:
%$C({\ast,\dotsb,\ast})$ denotes a positive constant which quantitatively 
%depends on the quantities indicated by the $\ast$'s but whose actual
%value can change from one occurrence to another, even within a single line.

 In this section, we recall the setting from \cite{MR3895752}.
 Our results are based on 
quantitative estimates and absorption arguments, where
it is often crucial to track the dependencies of constants quantitatively.
For this purpose, we will use the following notational convention:
$C({\ast,\dotsb,\ast})$ denotes a positive constant which quantitatively 
depends on the quantities indicated by the $\ast$'s but whose actual
value can change from one occurrence to another, even within a single line. 

\subsection{Metric spaces}\label{s.metric}
Unless otherwise specified, we assume that $X=(X,d,\mu)$ is a metric measure space equipped with a metric $d$ and a 
positive complete Borel
measure $\mu$ such that $0<\mu(B)<\infty$
for all balls $B\subset X$, each of which is always an open set of the form \[B=B(x,r)=\{y\in X\,:\, d(y,x)<r\}\] with $x\in X$ and $r>0$.
As in \cite[p.~2]{MR2867756},
we extend $\mu$ as a Borel regular (outer) measure on $X$.
We remark that the space $X$ is separable
under these assumptions, see \cite[Proposition 1.6]{MR2867756}.
We  also assume that $\# X\ge 2$ and  
that the measure $\mu$ is {\em doubling}, that is,
there is a constant $c_\mu> 1$, called
the {\em doubling constant of $\mu$}, such that
\begin{equation}\label{e.doubling}
\mu(2B) \le c_\mu\, \mu(B)
\end{equation}
for all balls $B=B(x,r)$ in $X$. 
Here we use for $0<t<\infty$ the notation $tB=B(x,tr)$. 
In particular, for all balls $B=B(x,r)$ that are centered at $x\in A\subset X$ with radius
$r\le \mathrm{diam}(A)$, we have that
\begin{equation}\label{e.radius_measure}
\frac{\mu(B)}{\mu(A)}\ge 2^{-s}\bigg(\frac{r}{\mathrm{diam}(A)}\bigg)^s\,,
\end{equation}
where $s=\log_2 c_\mu>0$. We refer to \cite[p.~31]{MR1800917}.
If $X$ is connected, then the doubling measure $\mu$ is also
{\em reverse doubling} in the sense that
there is a constant  $0<c_R=C(c_\mu)<1$ such that
\begin{equation}\label{e.rev_dbl_decay}
\mu(B(x,r/2))\le c_R\, \mu(B(x,r))
\end{equation}
for every $x\in X$ and %ball $B(x,r)\subset X$
$0<r<\diam(X)/2$.
See for instance~\cite[Lemma~3.7]{MR2867756}.

%In  the sequel, we will call $X$ satisfying the aforemention condition as a {\em metric space}.

\subsection{Geodesic spaces}
Let $X$ be a metric space satisfying the conditions stated in Section \ref{s.metric}.
By a {\em curve} we mean a nonconstant, rectifiable, continuous
mapping from a compact interval of $\R$ to $X$;  we tacitly assume
that all curves are parametrized by their arc-length.
We say that $X$ is a {\em geodesic space}, if 
every pair of points in $X$
can be joined by a curve whose length is equal to the distance between the two points. 
In particular, it easily follows that
\begin{equation}\label{e.diams}
0<\diam(2B)\le 4\diam(B)
\end{equation}
for all balls $B=B(x,r)$ in a geodesic space $X$.
 Since  geodesic spaces are connected, the measure $\mu$ is
reverse doubling in a geodesic space $X$ in the sense that
inequality \eqref{e.rev_dbl_decay} holds.

The following lemma is \cite[Lemma 12.1.2]{MR3363168}.
%see also \cite[Proposition 2.2.2]{MR2415381}.

\begin{lemma}\label{l.continuous}
Suppose that $X$ is a geodesic space and $A\subset X$ is a measurable set. 
Then the function
\[
r\mapsto \frac{\mu(B(x,r)\cap A)}{\mu(B(x,r))}\,:\, (0,\infty)\to \R
\]
is continuous whenever $x\in X$.
\end{lemma}

%When $A\subset X$, we let $\ol A$ denote the closure of $A$,
%and hence $\ol B$ always refers to the closure of the ball $B$, 
%not to the corresponding closed ball.

The second  
lemma, in turn, is \cite[Lemma 2.5]{MR3895752}.

\begin{lemma}\label{l.ball_measures}
Suppose that $B=B(x,r)$ and $B'=B(x',r')$ are two balls in a geodesic space $X$ such
that $x'\in B$ and $0<r'\le \mathrm{diam}(B)$. Then
$\mu(B')\le c_\mu^3 \mu(B'\cap B)$.
\end{lemma}

%\begin{proof}
%It suffices to find $y\in X$ such that
%$B(y,r'/4)\subset B'\cap B$.
%Inequality $\mu(B')\le c_\mu^3 \mu(B'\cap B)$ then follows
%from the doubling condition \eqref{e.doubling} and the fact that $B'\subset B(y,2r')$.
%
%Assume first that $x\in B(x',r'/4)$. In this case we may choose $y=x'$, 
%since we have for all $z\in B(x',r'/4)$ that
%\[
%d(z,x)\le d(z,x')+d(x',x)< r'/4+r'/4=r'/2\le \diam(B)/2\le r\,,
%\]
%and hence $B(x',r'/4)\subset B'\cap B(x,r)=B'\cap B$.
%
%Let us then consider the case $x\not\in B(x',r'/4)$. Since $X$ is  a geodesic space, there exists an arc-length parametrized curve $\gamma\colon [0,\ell]\to X$ with
%$\gamma(0)=x'$, $\gamma(\ell)=x$ and $\ell=d(x,x')$.
%We claim that $y=\gamma(r'/4)$ satisfies the required condition $B(y,r'/4)\subset B'\cap B$.
%In order to prove the claim, let us fix a point $z\in B(y,r'/4)$. Then
%\begin{align*}
%d(z,x')\le d(z,y)+d(y,x') <  r'/4 + d(\gamma(r'/4),\gamma(0))\le r'/2<r'\,. 
%\end{align*}
%Hence $z\in B(x',r')$ and therefore $B(y,r'/4)\subset B(x',r')=B'$. Moreover, since
%$\ell=d(x,x')$, 
%\begin{align*}
%d(z,x)&\le d(z,y)+d(y,x)
%<r'/4 + d(\gamma(r'/4),\gamma(\ell))\\
%&\le r'/4+ (\ell - r'/4)=\ell=d(x,x')<r\,.
%\end{align*}
%It follows that $z\in B(x,r)=B$ and therefore $B(y,r'/4)\subset B'\cap B$.
%\end{proof}

\subsection{H\"older and Lipschitz functions}
Let $A\subset X$.
We say that
$u\colon A\to \R$ is  a {\em $\beta$-H\"older function,} with 
an exponent $0<\beta\le 1$ and a constant 
$0\le \kappa <\infty$, if
\[
\lvert u(x)-u(y)\rvert\le \kappa\, d(x,y)^\beta\qquad \text{ for all } x,y\in A\,.
\]
If $u\colon A\to \R$ is a $\beta$-H\"older function, with a constant $\kappa$, then the classical McShane extension
\begin{equation}\label{McShane}
v(x)=\inf \{ u(y) + \kappa \,d(x,y)^\beta\,:\,y\in A\}\,,\qquad x\in X\,,
\end{equation}
defines a $\beta$-H\"older function $v\colon X\to \R$,
with the constant $\kappa$, 
which satisfies
$v|_A = u$; we refer to \cite[pp.~43--44]{MR1800917}.
The set of all $\beta$-H\"older functions $u\colon A\to\R$
is denoted by $\Lip_\beta(A)$. 
The $1$-H\"older functions are also called {\em Lipschitz functions}.
%We denote $\Lip(A)=\Lip_1(A)$.

\subsection{Additional notation}

 We write $\N=\{1,2,3,\ldots\}$ and $\N_0=\N\cup\{0\}$. We use the following familiar notation: \[
u_A=\vint_{A} u(y)\,d\mu(y)=\frac{1}{\mu(A)}\int_A u(y)\,d\mu(y)
\]
is the integral average of $u\in L^1(A)$ over a measurable set $A\subset X$
with $0<\mu(A)<\infty$. 
The characteristic function of set $A\subset X$ is denoted by $\mathbf{1}_{A}$; that is, $\mathbf{1}_{A}(x)=1$ if $x\in A$
and $\mathbf{1}_{A}(x)=0$ if $x\in X\setminus A$.
 The distance between a point $x\in X$ and a set $A\subset X$
is denoted by $d(x,A)$.
The closure of a set $A\subset X$ is denoted by $\iol{A}$. In particular, if $B\subset X$ is a ball,
then the notation $\iol{B}$ refers to the closure
of the ball $B$.

\section{Haj{\l}asz gradients}\label{s.abstract}

We work with Haj{\l}asz $\beta$-gradients of order $0<\beta\le 1$ in a metric space $X$. 

\begin{definition}
For each function $u\colon X\to \R$, we let 
$\mathcal{D}_H^{\beta}(u)$ 
be the (possibly empty) family of all measurable functions $g\colon X\to [0,\infty]$ such that
\begin{equation}\label{e.hajlasz}
\lvert u(x)-u(y)\rvert \le d(x,y)^\beta\big( g(x)+g(y) \big)
\end{equation}
almost everywhere, i.e., there exists an exceptional set $N=N(g)\subset X$ for which $\mu(N)=0$ and
inequality \eqref{e.hajlasz} holds for every $x,y\in X\setminus N$. 
A function  $g\in\mathcal{D}^\beta_H(u)$ is called a Haj{\l}asz $\beta$-gradient of the function $u$.
\end{definition}

The Haj{\l}asz $1$-gradients in metric spaces are introduced in \cite{MR1401074}. 
More details on these gradients and their applications can be found, for instance, from \cite{MR1876253,MR1681586,MR3357989,MR1859234,MR2025934}.
%We remark that $\mathcal{D}^\beta_H(u)$ is nonempty if $u$ is a $\beta$-H\"older function in $X$. 
The following basic properties are easy to verify for all $\beta$-H\"older functions $u,v\colon X\to \R$
\begin{itemize}
\item[(D1)]  $\lvert a\rvert g\in \mathcal{D}_H^{\beta}(au)$ if $a\in\R$ and $g\in \mathcal{D}_H^{\beta}(u)$;
\item[(D2)] $g_u + g_v\in \mathcal{D}_H^{\beta}(u+v)$ if $g_u\in\mathcal{D}_H^{\beta}(u)$ and $g_v\in\mathcal{D}_H^{\beta}(v)$;
\item[(D3)]  If $f\colon \R \to \R$ is a Lipschitz function with constant $\kappa$, then
$\kappa g\in\mathcal{D}^\beta_H(f\circ u)$ if $g\in\mathcal{D}^\beta_H(u)$.
\end{itemize}

There are both disadvantages and advantages to working with
Haj{\l}asz gradients.
A technical disadvantage is their nonlocality \cite{MR1876253}.
For instance, if $u$ is constant on some set $A\subset X$ and $g\in\mathcal{D}^\beta_H(u)$, then
$g\mathbf{1}_{X\setminus A}$ need not belong to $\mathcal{D}^\beta_H(u)$.
By the so-called glueing lemma, see for instance \cite[Lemma 2.19]{MR2867756}, the corresponding localization property holds for so-called $p$-weak upper gradients,
which makes their application more flexible.
However, the following nonlocal glueing lemma 
from \cite[Lemma 6.6]{MR3895752}
holds in the setting of Haj{\l}asz gradients. 

We recall the proof for convenience.

\begin{lemma}\label{l.glueing}
Let $0<\beta\le 1$ and let $A\subset X$ be a Borel set.
Let $u\colon X\to \R$ be a $\beta$-H\"older function and
suppose that $v\colon X\to \R$ 
is such that $v|_{X\setminus A} =u|_{X\setminus A}$ and
there exists a constant $\kappa\ge 0$ such that
$\lvert v(x)-v(y)\rvert \le \kappa\, d(x,y)^\beta$
for all $x,y\in X$.
 Then
\[
g_v=\kappa\, \mathbf{1}_{A} + g_u\mathbf{1}_{X\setminus A} \in \mathcal{D}_{H}^{\beta}(v)
\]
whenever $g_u\in \mathcal{D}_H^{\beta}(u)$.
\end{lemma}

\begin{proof}
Fix a function $g_u\in \mathcal{D}_H^{\beta}(u)$ and let 
$N\subset X$ be the exceptional set such that $\mu(N)=0$ and 
inequality \eqref{e.hajlasz} holds for every $x,y\in X\setminus N$ and with $g=g_u$.

Fix $x,y\in X\setminus N$. If $x,y\in X\setminus A$, then
\[
\lvert v (x)-v(y)\rvert =\lvert u (x)-u(y)\rvert 
\le d(x,y)^\beta \big(g_u(x)+g_u(y)\big) = d(x,y)^\beta \big(g_v (x)+g_v (y)\big)\,.
\]
If $x\in A$ or $y\in A$, then
\[
\lvert v (x)-v(y)\rvert \le \kappa\, d(x,y)^\beta\le d(x,y)^\beta\big(g_v (x)+g_v (y)\big)\,.
\]
By combining the estimates above, we find that
\[
\lvert v (x)-v(y)\rvert \le d(x,y)^\beta\big(g_v (x)+g_v (y)\big)
\]
whenever $x,y\in X\setminus N$. The desired conclusion $g_v\in\mathcal{D}_H^{\beta}(v)$ follows.
\end{proof}

The following nonlocal generalization of the Leibniz rule is from \cite{MR1681586}.
The proof is recalled for the  convenience of the reader. The nonlocality is reflected 
by the appearence of the two global terms $\lVert \psi\lVert_\infty$ and $\kappa$ in the statement below.

\begin{lemma}\label{l.Leibniz}
Let $0<\beta\le 1$, let
 $u\colon X\to \R$ be a bounded $\beta$-H\"older function, and let $\psi\colon X\to \R$ be a bounded $\beta$-H\"older
function with a constant $\kappa\ge 0$.  Then $u\psi\colon X\to \R$ is a $\beta$-H\"older function  and 
\[
(g_u\lVert \psi\lVert_\infty + \kappa \lvert  u\rvert)\mathbf{1}_{\{\psi\not=0\}}\in\mathcal{D}_H^{\beta}(u\psi)
\]
for all $g_u\in\mathcal{D}_H^{\beta}(u)$.
Here $\{\psi\not=0\}=\{y\in X: \psi(y)\not=0\}$.
\end{lemma}

\begin{proof}
Fix $x,y\in X$. Then
\begin{equation}\label{e.basic_triangle}
\begin{split}
\lvert u(x)\psi(x) - u(y)\psi(y)\rvert&=\lvert u(x)\psi(x) - u(y)\psi(x) + u(y)\psi(x) - u(y)\psi(y)\rvert\\
&\le \lvert \psi(x)\rvert\lvert u(x) - u(y)\rvert + \lvert u(y)\rvert\lvert \psi(x) - \psi(y)\rvert\,.
\end{split}
\end{equation}
Since $u$ and $\psi$ are both bounded $\beta$-H\"older functions in $X$, it follows
that $u\psi$ is $\beta$-H\"older in $X$.

Fix a function $g_u\in \mathcal{D}_H^{\beta}(u)$ and let 
$N\subset X$ be the exceptional set such that $\mu(N)=0$ and 
inequality \eqref{e.hajlasz} holds for every $x,y\in X\setminus N$ and with $g=g_u$.
Denote $h=(g_u\lVert \psi\lVert_\infty + \kappa \lvert  u\rvert)\mathbf{1}_{\{\psi\not=0\}}$. Let $x,y\in X\setminus N$. It suffices to show that \[
\lvert u(x)\psi(x) - u(y)\psi(y)\rvert\le d(x,y)^\beta(h(x)+h(y))\,.\]
By \eqref{e.basic_triangle}, we get
\begin{equation}\label{e.toka_apu}
\begin{split}
\lvert u(x)\psi(x) - u(y)\psi(y)\rvert&\le \lvert \psi(x)\rvert d(x,y)^\beta (g_u(x)+g_u(y)) + \lvert u(y)\rvert \kappa d(x,y)^\beta\\
&= d(x,y)^\beta \left(\lvert \psi(x)\rvert(g_u(x)+g_u(y))+\kappa\lvert u(y)\rvert \right)\,.
\end{split}
\end{equation}
Next we do a case study. If  $x,y\in \{\psi\not=0\}$, then by \eqref{e.toka_apu} we have
\begin{align*}
\lvert u(x)\psi(x) - u(y)\psi(y)\rvert &\le d(x,y)^\beta \left(g_u(x) \lVert \psi\rVert_\infty \mathbf{1}_{\{\psi\not=0\}}(x)+  (g_u(y) \lVert \psi\rVert_\infty+\kappa\lvert u(y)\rvert) \mathbf{1}_{\{\psi\not=0\}}(y)\right)\\
&\le  d(x,y)^\beta(h(x)+h(y))\,.
\end{align*}
If $x\in X\setminus  \{\psi\not=0\}$  and $y\in \{\psi\not=0\}$, then
\begin{align*}
\lvert u(x)\psi(x) - u(y)\psi(y)\rvert &\le d(x,y)^\beta \left(\kappa\lvert u(y)\rvert \mathbf{1}_{\{\psi\not=0\}}(y)\right)\\
&=d(x,y)^\beta h(y)\le  d(x,y)^\beta(h(x)+h(y))\,.
\end{align*}
The case $x\in \{\psi\not=0\}$  and $y\in X\setminus \{\psi\not=0\}$ is symmetric
and the last case is trivial.
\end{proof}

A significant  advantage of working with Haj{\l}asz gradients is that Poincar\'e inequalities  are always valid \cite{MR1859234,MR2025934}. The same is not true for the usual $p$-weak upper
gradients, in which case a Poincar\'e inequality often has to be assumed.

The following theorem gives a $(\beta,p,p)$-Poincar\'e inequality for any $1\le p<\infty$.
This inequality relates the Haj{\l}asz gradient to the given measure. 

\begin{theorem}\label{t.pp_poincare}
Suppose that $X$ is a metric space.
Fix  exponents
%Suppose that we are given a $\mathcal{D}$-structure
%in a geodesic space $X$, with exponents 
$1\le p<\infty$ and $0<\beta \le 1$.
Suppose that 
$u\in\mathrm{Lip}_\beta(X)$ and that $g\in \mathcal{D}_H^{\beta}(u)$.
Then
\[
\bigg(\vint_B \lvert u(x)-u_B\rvert^p\,d\mu(x)\bigg)^{1/p}\le 2 \diam(B)^\beta\bigg(\vint_B g(x)^p\,d\mu(x)\bigg)^{1/p}
\]
holds whenever $B\subset X$ is a ball.
\end{theorem}

\begin{proof}
We follow the  proof of \cite[Theorem 5.15]{MR1800917}.
Let 
$N=N(g)\subset X$ be the exceptional set such that $\mu(N)=0$ and 
 \eqref{e.hajlasz} holds for every $x,y\in X\setminus N$.
By H\"older's inequality
\[
\vint_B \lvert u(x)-u_B\rvert^p\,d\mu(x)\le \vint_{B\setminus N} \vint_{B\setminus N} \lvert u(y)-u(x)\rvert^p\,d\mu(y)\,d\mu(x)\,.
\]
Applying \eqref{e.hajlasz}, we obtain
\begin{align*}
&\vint_{B\setminus N} \vint_{B\setminus N} \lvert u(y)-u(x)\rvert^p\,d\mu(y)\,d\mu(x)
\le \vint_{B\setminus N} \vint_{B\setminus N} d(x,y)^{\beta p} \left( g(x) + g(y) \right)^p\,d\mu(y)\,d\mu(x)\\
&\quad \le 2^{p-1}\diam(B)^{\beta p} \vint_{B\setminus N} \vint_{B\setminus N} \left(g(x)^p + g(y)^p\right)\,d\mu(y)\,d\mu(x)
\le 2^p\diam(B)^{\beta p}\vint_B g(x)^p\,d\mu(x)\,.
\end{align*}
The claimed inequality follows by combining the above estimates.
\end{proof}

In a geodesic space,  even a stronger $(\beta,p,q)$-Poincar\'e inequality holds for some $q<p$. In the context of $p$-weak upper gradients, this
result corresponds to the deep theorem of Keith and Zhong \cite{MR2415381}.
In our context the proof is simpler, since we have $(\beta,q,q)$-Poincar\'e inequalities for all exponents $1<q<p$
 by Theorem \ref{t.pp_poincare}. It remains to argue that one of these inequalities self-improves to  a $(\beta,p,q)$-Poincar\'e inequality when $q<p$ is sufficiently close to $p$.

\begin{theorem}\label{t.qp_poincare}
Suppose that $X$ is a geodesic space.
Fix  exponents
%Suppose that we are given a $\mathcal{D}$-structure
%in a geodesic space $X$, with exponents 
$1< p<\infty$ and $0<\beta \le 1$.
Suppose that 
$u\in\mathrm{Lip}_\beta(X)$ and that $g\in \mathcal{D}_H^{\beta}(u)$.
Then there exists an exponent $1<q<p$ and a constant $C$, both depending on $c_\mu$, $p$ and $\beta$, such that
\[
\bigg(\vint_B \lvert u(x)-u_B\rvert^p\,d\mu(x)\bigg)^{1/p}\le C \diam(B)^\beta\bigg(\vint_B g(x)^q\,d\mu(x)\bigg)^{1/q}
\]
holds whenever $B\subset X$ is a ball.
\end{theorem}

\begin{proof}
Fix $Q=Q(\beta,p,c_\mu)$ such that $Q>\max\{\log_2 c_\mu,\beta p\}$. Since
\[\lim_{q\to p} Qq/(Q-\beta q)=Qp/(Q-\beta p)>p\,,\] there exists
$1<q=q(\beta,p,c_\mu)<p$ such that $p<Qq/(Q-\beta q)$ and $\beta q<Q$. Theorem \ref{t.pp_poincare} and H\"older's inequality implies that
\[
\vint_B \lvert u(x)-u_B\rvert\,d\mu(x)\le \left(\vint_B \lvert u(x)-u_B\rvert^q\,d\mu(x)\right)^{1/q}\le 2 \diam(B)^\beta\bigg(\vint_B g(x)^q\,d\mu(x)\bigg)^{1/q}
\]
whenever $B\subset X$ is a ball. Now the claim follows
from \cite[Theorem 3.6]{MR3895752}, which is based on the covering argument from \cite{MR1336257}.
We also refer to \cite[Lemma 2.2]{MR3089750}.
\end{proof}

\section{Capacity density condition}\label{s.capacity}

In this section we define the capacity density condition. This condition is based
on the following notion of variational capacity, and it is weaker than the well known measure density condition. We also prove 
boundary Poincar\'e inequalities for sets satisfying a capacity density condition. This is done with the aid of so-called Maz{\cprime}ya's inequality, which provides
an important link between Poincar\'e inequalities and capacities.

\begin{definition}\label{d.varcap}
Let $1\le p<\infty$, $0<\beta\le 1$, and let $\Omega\subset X$ be an open set. 
The variational $(\beta,p)$-capacity of a subset $F\subset\Omega$ with  $\mathrm{dist}(F,X\setminus \Omega)>0$ is \begin{equation*}\label{e.cap}
\cp_{\beta,p}(F,\Omega)=\inf_u\inf_g\int_X g(x)^p\,d\mu(x)\,,
\end{equation*}
where the infimums are taken  over all $\beta$-H\"older functions $u$ in $X$, with $u\ge 1$ in $F$ and $u=0$ in $X\setminus \Omega$, and over all $g\in\mathcal{D}_H^{\beta}(u)$.
\end{definition}

\begin{remark}\label{r.varcap_properties} 
We may take the infimum in Definition~\ref{d.varcap} among all $u$ satisfying additionally $0\le u\le 1$. This follows by considering the $\beta$-H\"older function function $v=\max\{0,\min\{u,1\}\}$ since $g\in\mathcal{D}^\beta_H(v)$ by Property~(D3).
\end{remark}

\begin{definition}\label{d.cap_density}
A closed set $E\subset X$ satisfies a $(\beta,p)$-capacity density condition, 
for $1\le p<\infty$ and $0<\beta\le 1$, if there
exists a constant $c_0>0$ such that
\begin{equation}\label{e.fatness}
\cp_{\beta,p}(E\cap \overline{B(x,r)},B(x,2r))\ge c_0 r^{-\beta p}\mu(B(x,r))
\end{equation}
for all $x\in E$ and all $0<r<(1/8)\diam(E)$.
%If there exists a constant $r_0>0$ such that
%condition~\eqref{e.fatness} holds for all
%$x\in E$ and all $0<r<r_0$, the closed set $E$ is said to be
%{\em locally uniformly $p$-fat}.
\end{definition}

\begin{example}\label{e.meas_density}
We say that a closed set $E\subset X$ satisfies a measure density condition, if there exists
a constant $c_1$ such that
\begin{equation}\label{e.meas_density_def}
\mu(E\cap \overline{B(x,r)})\ge c_1 \mu(B(x,r))
\end{equation}
for all $x\in E$ and all $0<r<(1/8)\diam(E)$.
Assume that the metric space $X$ is connected, $1\le p<\infty$ and $0<\beta\le 1$, and that  a set $E\subset X$ satisfies
a measure density condition. Then it is easy to show that $E$ satisfies a $(\beta,p)$-capacity density condition, see below. We remark that the  measure density condition has been applied in 
\cite{KilpelainenKinnunenMartio2000} to study Haj{\l}asz Sobolev spaces
with zero boundary values on $E$.

Fix $x\in E$ and $0<r<(1/8)\diam(E)$. We aim to show that \eqref{e.fatness} holds. For this purpose, we write $F=E\cap \overline{B(x,r)}$ and $B=B(x,r)$.
Let $u\in \Lip_\beta(X)$ be such that  $0\le u\le 1$, $u= 1$ in $F$ and $u=0$ in $X\setminus 2B$.
Let also $g\in\mathcal{D}_H^{\beta}(u)$.
Recall that $X$ is connected. Hence, by  the properties of $u$ and the reverse doubling inequality~\eqref{e.rev_dbl_decay},  we obtain
\begin{align*}
0\le u_{4B}=\vint_{4B} u(y)\,d\mu(y) \le \frac{\mu(2B)}{\mu(4 B)}
\le c_R < 1.
\end{align*}
If $y\in F$, we have $v(y)=1$ and therefore
\[
\lvert u(y)- v_{4B}\rvert  \ge 1-u_{4B}\ge 1-c_R=C(c_\mu)>0.
\]
Applying the measure density condition \eqref{e.meas_density_def} and the  $(\beta,p,p)$-Poincar\'e inequality, see Theorem~\ref{t.pp_poincare}, we obtain
\begin{align*}
c_1\mu(B)&\le \mu(F)\le C(c_\mu,p)\int_{F} \lvert u(y)- u_{4B}\rvert^p\,d\mu(y)\\
&\le C(c_\mu,p)\int_{4B} \lvert u(y)- u_{4B}\rvert^p\,d\mu(y)\\&\le C(c_\mu,p)r^{\beta p} \int_{4B} g(y)^p\,d\mu(y)
\le C(c_\mu,p)r^{\beta p} \int_{X} g(y)^p\,d\mu(y)\,.
\end{align*}
By taking infimum over functions $u$ and $g$ as above, we see that
\[
\mathrm{cap}_{\beta,p}(E\cap \overline{B(x,r)},2B)=\mathrm{cap}_{\beta,p}(F,2B)\ge C(c_1,c_\mu,p) r^{-\beta p}\mu(B)\,.
\]
This shows that $E$ satisfies a $(\beta,p)$-capacity density condition \eqref{e.fatness}.
\end{example}

The following Maz{\cprime}ya's inequality
provides a link between capacities and Poincar\'e inequalities.
We refer to \cite[Chapter~10]{Mazya1985} and \cite[Chapter~14]{Mazya2011}
for further details on such inequalities.

\begin{theorem}\label{t.Mazya_Poincare}
Let $1\le p<\infty$, $0<\beta\le 1$, and let $B(z,r)\subset X$ be a ball.
Assume that $u$ is a $\beta$-H\"older function in $X$ and $g\in\mathcal{D}_H^{\beta}(u)$.
Then there exists a constant $C=C(p)$ such that
\[
\vint_{B(z,r)} \lvert u(x)\rvert^{p}\,d\mu(x)
\le \frac{C}{\cp_{\beta,p}\bigl(\{u=0\}\cap\overline{B(z,\frac r 2)},B(z,r)\bigr)}\int_{B(z,r)} g(x)^p\,d\mu(x)\,.
\]
Here $\{u=0\}=\{y\in X : u(y)=0\}$. 
\end{theorem}

\begin{proof} 
We adapt the proof of \cite[Theorem 5.47]{KLV2021}, which in turn is based on \cite[Theorem 5.53]{MR2867756}. 
Let $M=\sup\{\lvert u(x)\rvert : x\in B(z,r)\}<\infty$.
By considering $\min\{M,\lvert u\rvert\}$ instead of $u$ and using (D3), we may assume that $u$ is a bounded $\beta$-H\"older function in $X$ and that $u\ge 0$ in $B(z,r)$. 
Write $B=B(z,r)$ and
\[
u_{B,p}=\biggl(\vint_{B} u(x)^p\,d\mu(x)\biggr)^{\frac1p} = \mu(B)^{-\frac1p}\lVert u \rVert_{L^p(B)}<\infty\,.
\]
If $u_{B,p} = 0$ the claim is true, and thus we may assume that $u_{B,p} > 0$.
Let
\[
\psi(x)=\max \Bigl\{0,1-(2r^{-1})^{\beta}d\bigl(x,B(z,\tfrac r2)\bigr)^\beta\Bigr\}\,
\]
for every $x\in X$. 
Then $0\le\psi\le 1$, $\psi=0$ in $X\setminus B(z,r)$, $\psi=1$ in $\overline{B(z,\tfrac r2)}$,
and  $\psi$ is a $\beta$-H\"older function in $X$ with a constant $(2r^{-1})^\beta$.
Let
\[
v(x)=\psi(x)\biggl(1-\frac {u(x)} {u_{B,p}}\biggr)\,,\qquad x\in X\,.
\]
Then 
$v=1$ in $\{u=0\}\cap \overline{B(z,\tfrac r2)}$ and $v=0$ in $X\setminus B(z,r)$.
By Lemma \ref{l.Leibniz},  and properties (D1) and (D2), the function $v$ is $\beta$-H\"older in $X$ and   
\[
g_v=\left(\frac{g}{u_{B,p}}\lVert \psi\rVert_\infty + (2r^{-1})^\beta \bigg\lvert 1-\frac {u} {u_{B,p}}\bigg\rvert\right)\mathbf{1}_{\{\psi\not=0\}}\in\mathcal{D}_H^{\beta}(v)\,.
\]
 Here we used the fact that $g\in\mathcal{D}_H^{\beta}(u)$ by assumptions. Now, the pair $v$ and $g_v$ is admissible for testing the capacity. Thus, we obtain
\begin{equation}\label{e.cap_of_null}
\begin{split}
&\cp_{\beta,p}\bigl(\{u=0\}\cap\overline{B(z,\tfrac r2)},B(z,r)\bigr)
 \le \int_{X} g_v(x)^p\,d\mu(x)\\
&\qquad\le \frac{C(p)}{(u_{B,p})^p} \int_{B} g(x)^p\,d\mu(x) + \frac {C(p)}{r^{\beta p} (u_{B,p})^p} \int_{B}\lvert u(x)-u_{B,p}\rvert^p\,d\mu(x)\,.
\end{split}
\end{equation}
We use Minkowski's inequality and  the $(\beta,p,p)$-Poincar\'e inequality in Theorem \ref{t.pp_poincare}
to estimate the second term on the right-hand side of~\eqref{e.cap_of_null}, and obtain
\begin{equation}\label{e.iol_u_est}
\begin{split}
&\biggl(\vint_{B}\lvert u(x)-u_{B,p}\rvert^p\,d\mu(x)\biggr)^{\frac 1p}
\le \biggl(\vint_{B}\lvert u(x)-u_{B}\rvert^p\,d\mu(x)\biggr)^{\frac1p} + \lvert u_{B,p} - u_{B}\rvert\\
&\qquad \le Cr^\beta\biggl(\vint_{B} g(x)^p\,d\mu(x)\biggr)^{\frac 1p}
  + \mu(B)^{-\frac1p}\bigl\lvert \lVert u \rVert_{L^p({B})} - \lVert u_{B} \rVert_{L^p({B})} \bigr\rvert.
\end{split}
\end{equation}
By the triangle inequality and the above  Poincar\'e inequality,
we have
\begin{equation*}
\begin{split}
\mu(B)^{-\frac1p}\bigl\lvert \lVert u \rVert_{L^p({B})} - \lVert u_{B} \rVert_{L^p({B})} \bigr\rvert
& \le \mu(B)^{-\frac1p}\lVert u - u_{B}\rVert_{L^p({B})}\\
& = \biggl(\vint_{B}\lvert u(x)-u_{B}\rvert^p\,d\mu(x)\biggr)^{\frac 1p}\\
&\le Cr^\beta\biggl(\vint_{B} g(x)^p\,d\mu(x)\biggr)^{\frac 1p}.
\end{split}
\end{equation*} 
Together with~\eqref{e.iol_u_est} this gives
\[
\biggl(\vint_{B}\lvert u(x)-u_{B,p}\rvert^p\,d\mu(x)\biggr)^{\frac 1p} 
\le Cr^\beta\biggl(\vint_{B} g(x)^p\,d\mu(x)\biggr)^{\frac 1p}, 
\]
and thus
\[
\int_{B}\lvert u(x)-u_{B,p}\rvert^p\,d\mu(x)
\le C(p)r^{\beta p} \int_{B} g(x)^p\,d\mu(x).
\]
Substituting this to~\eqref{e.cap_of_null} and recalling that $B=B(z,r)$, we arrive at
\[
\cp_{\beta,p}\bigl(\{u=0\}\cap\overline{B(z,\tfrac r2)},B(z,r)\bigr)
 \le \frac{C(p)}{(u_{B,p})^p} \int_{B(z,r)} g(x)^p\,d\mu(x). 
\]
The claim follows by reorganizing the terms.
\end{proof}

The following theorem establishes certain boundary Poincar\'e inequalities
for a set $E$ satisfying a capacity density condition. Maz{\cprime}ya's inequality in Theorem \ref{t.Mazya_Poincare}
is a key tool in the proof.

\begin{theorem}\label{t.global_boundary_poincare}
Let $1\le p<\infty$ and $0<\beta\le 1$.
Assume that  $E\subset X$ satisfies a $(\beta,p)$-capacity density condition
with a constant $c_0$.
Then there is a constant $C=C(p,c_0,c_\mu)$ 
such that
\begin{equation}\label{eq.bdry_poinc}
\intav_{B(x,R)} \lvert u(x)\rvert^p\,d\mu(x)
\le C R^{\beta p}\intav_{B(x,R)} g(x)^p\,d\mu(x)
%=Cr^{sp}\intav_{2\lambda B} 
%\jl g_{u,s,t,2\lambda B}(x)^p\ed\,d\mu(x)
%\biggl(\int_{2\lambda B}\frac{|u(x)-u(y)|^t}{d(x,y)^{st}\mu(B(x,d(x,y)))}\,dy\biggr)^{p}\,d\mu(x)
\end{equation}
whenever $u\colon X\to \R$ is a $\beta$-H\"older function  in $X$ 
such that  %$\lvert u\rvert^q$ is integrable on balls, 
$u=0$ in $E$, $g\in\mathcal{D}_H^{\beta}(u)$, and $B(x,R)$ is a ball with
$x\in E$ and  $0<R<\diam(E)/4$.
\end{theorem}

\begin{proof}
Let 
%Fix a ball $B=B(x,r)\subset X$ such that
$x\in E$ and $0<R<\diam(E)/4$. 
%Then $R<\diam(X)/2$.
We denote $r=R/2<\diam(E)/8$. 
Applying the capacity density condition in the ball $B=B(x,r)$ gives
\begin{align*}
\cp_{\beta,p}(E\cap \iol{B},2B)\ge c_0r^{-\beta p}\mu(B)\,.
\end{align*}
Write $\{u=0\}=\{y\in X :  u(y)=0\}\supset E$. By the monotonicity of capacity and the doubling condition we have
\[
\frac{1}{\cp_{\beta,p}(\{u=0\}\cap \iol{B},2B)}
\le \frac{1}{\cp_{\beta,p}(E\cap \iol{B},2B)}
\le \frac{C(c_0) r^{\beta p}}{\mu(B)}\le \frac{C(c_0,c_\mu) R^{\beta p}}{\mu(2 B)}.
\]
The desired inequality, for the ball $B(x,R)=B(x,2 r)=2B$, follows from Theorem~\ref{t.Mazya_Poincare}.
\end{proof}

\section{Necessary and sufficient geometrical conditions}\label{s.geometry}

In this section we adapt the approach in \cite{DLV2021} by 
giving necessary and sufficient geometrical conditions for the $(\beta,p)$-capacity density condition.
These are given  in terms of the following upper Assouad codimension 
 \cite{KaenmakiLehrbackVuorinen2013}.

\begin{definition}
When $E\subset X$ and $r>0$, the
open $r$-neighbourhood of $E$
is the set 
\[E_r=\{x\in X : d(x,E)<r\}.\]
The upper Assouad codimension of $E\subset X$,
denoted by $\ucodima(E)$, is
the infimum of all $Q\ge 0$ for
which there is $c>0$ such that
\[
\frac{\mu(E_r\cap B(x,R))}{\mu(B(x,R))}\ge c\Bigl(\frac{r}{R}\Bigr)^Q
\]
for every $x\in E$ and all $0<r<R<\diam(E)$. 
If $\diam(E)=0$, then
the restriction $R<\diam(E)$ is removed.
\end{definition}

Observe that a larger
 set has a smaller Assouad codimension.
We need suitable versions of Hausdorff contents
from \cite{MR3631460}.

\begin{definition}
 The ($\rho$-restricted) Hausdorff content of codimension $q\ge 0$ 
of a set $F\subset X$ is defined by  
\[
\Ha^{\mu,q}_\rho(F)=\inf\Biggl\{\sum_{k} \mu(B(x_k,r_k))\,r_k^{-q} :
F\subset\bigcup_{k} B(x_k,r_k)\text{ and } 0<r_k\leq \rho \Biggr\}.
\]
\end{definition}

The following lemma is \cite[Lemma~5.1]{MR3631460}.
It provides a lower bound for the Hausdorff
content of a set truncated in a fixed ball in terms of the measure and radius of the truncating ball.
The proof uses completeness  via construction of a compact Cantor-type set inside $E$, to which
mass is uniformly distributed by a Carath\'eodory construction.

\begin{lemma}\label{l.measure_to_hausdorff}
Assume that $X$ is a complete metric space.
Let $E\subset X$ be a closed set, and assume that
$\ucodima(E)<q$. Then there exists a constant $C>0$ such that
\begin{equation}\label{e.hausdorff_content_density}
\Ha^{\mu,q}_r(E\cap \overline{B(x,r)})\ge Cr^{-q} \mu(B(x,r))
\end{equation}
for every $x\in E$ and all  $0<r<\diam(E)$.
\end{lemma}

On the other hand, Hausdorff contents gives a lower bound for capacity by following lemma.
The proof is based on a covering argument, where the
covering balls are chosen by chaining. 
The proof is a more sophisticated variant of the argument given in Example \ref{e.meas_density}.
Similar
covering arguments via chaining have been widely used; see for instance \cite{HeinonenKoskela1998}. \ed

\begin{lemma}\label{l.codim_sufficient}
Assume that $X$ is a connected metric space. 
Let $0<\beta\le 1$,  $1\le  p<\infty$, and $0< \eta<p$. 
Assume that $B=B(x_0,r)\subset X$ is a ball with $r<\diam(X)/8$,
and assume that $F\subset \iol{B}$ is a closed set.
Then there is a constant $C=C(\beta,p,\eta,c_\mu)>0$ such that
\[
r^{\beta(p-\eta)}\cp_{\beta,p}(F,2B)\ge C\mathcal{H}^{\mu,\beta\eta}_{20r}(F)\,.
\]
\end{lemma}

\begin{proof}
We adapt the proof of \cite[Lemma 4.6]{DLV2021} for our purposes. 
Let $u\in \Lip_\beta(X)$ be such that $0\le u\le 1$ in $X$, $u=1$ in $F$ and $u=0$ in $X\setminus 2B$.
Let also $g\in\mathcal{D}_H^{\beta}(u)$.
We aim to cover the set $F$ by balls that are chosen by chaining. In order to do so, we fix $x\in F$ and write $B_0=4B=B(x_0,4 r)$,
$r_0=4 r$, $r_j=2^{-j+1}r$ and $B_j=B(x,r_j)$, $j=1,2,\ldots$.
Observe that $B_{j+1}\subset B_j$ and $\mu(B_j)\le c_\mu^3 \mu(B_{j+1})$ if
$j=0,1,2,\ldots$.

By  the above properties of $u$ and the reverse doubling inequality \eqref{e.rev_dbl_decay}, we obtain
\begin{align*}
0\le u_{B_0}=\vint_{B_0} u(y)\,d\mu(y) \le \frac{\mu(2B)}{\mu(4 B)}
\le c_R < 1.
\end{align*}
Since $x\in F$, we find that $u(x)=1$ and therefore
\[
\lvert u(x)- u_{B_0}\rvert  \ge 1-u_{B_0}\ge 1-c_R=C(c_\mu)>0.
\]
We write $\delta=\beta(p-\eta)/p>0$.
Using the Poincar\'e inequality in Theorem \ref{t.pp_poincare} and abbreviating $C=C(\beta,p,\eta,c_\mu)$,  we obtain 
\begin{align*}
\sum_{j=0}^\infty 2^{-j\delta}&=C(1-c_R)
\le C \lvert u(x)-u_{B_0}\rvert\\&\le C\sum_{j=0}^\infty \lvert u_{B_{j+1}}-u_{B_j}\rvert\le C\sum_{j=0}^\infty \frac{\mu(B_j)}{\mu(B_{j+1})} \vint_{B_j} \lvert u(y)-u_{B_j}\rvert\,d\mu(y)
\\& \le C\sum_{j=0}^\infty \left(\vint_{B_j} \lvert u(y)-u_{B_j}\rvert^p\,d\mu(y)\right)^{\frac{1}{p}}\le C\sum_{j=0}^\infty  r_j^{\beta} 
\biggl(\intav_{B_j} g(y)^p\,d\mu(y)\biggr)^{\frac{1}{p}}\,.
\end{align*}
%where we have written 
%\[
%g_\varphi(y)=\biggl(\int_{2B} \frac{\vert \varphi(y)-\varphi(z)\vert^t}{d(y,z)^{st}\mu(B(y,d(y,z)))}\,dz\biggr)^{1/t},\quad y\in X.
%\]
By comparing the series in the left- and right-hand side of these  inequalities, we see that there exists $j\in \{0,1,2,\ldots\}$
depending on $x$ such that
\begin{equation}\label{e.covering_ineq}
2^{-j\delta p}\le C(\beta,p,\eta,c_\mu) r_{j}^{\beta p} \intav_{B_{j}} g(y)^p\,d\mu(y) .
\end{equation}
Write $r_x=r_{j}$ and
 $B_x=B_{j}$. Then $x\in B_x$ and straightforward estimates
based on \eqref{e.covering_ineq} give
\[
 \mu(B_x) r_x^{-\beta\eta}\le C(\beta,p,\eta,c_\mu)  r^{\beta(p-\eta)} \int_{B_x}  g(y)^p\,d\mu(y)\,.
\]

By the $5r$-covering lemma \cite[Lemma~1.7]{MR2867756}, we obtain points $x_k\in F$,
$k=1,2,\ldots$, such that the  balls $B_{x_k}\subset B_0=4B$ with radii $r_{x_k}\le 4r$  are
pairwise disjoint and 
$F\subset \bigcup_{k=1}^\infty 5B_{x_k}$. Hence,
\begin{align*}
\mathcal{H}^{\mu,\beta\eta}_{20 r}(F) &\le \sum_{k=1}^\infty  \mu(5B_{x_k})
(5r_{x_k})^{-\beta\eta}%\\&
 \le  C \sum_{k=1}^\infty  r^{\beta(p-\eta)} \int_{B_{x_k}}g(x)^p\,d\mu(x)\\
&\le  C r^{\beta(p-\eta)} \int_{4 B}g(x)^p\,d\mu(x)
\le C r^{\beta(p-\eta)} \int_{X}g(x)^p\,d\mu(x)\,,
\end{align*}
where  $C=C(\beta,p,\eta,c_\mu)$. We remark that the scale $20r$ of the Hausdorff content in the left-hand side comes
from the fact that radii of the covering balls $5B_{x_k}$ for $F$ are bounded by $20r$.
The desired inequality follows by taking infimum over all functions
$g\in\mathcal{D}_H^{\beta}(u)$ and then over all functions $u$ as above.
\end{proof}
%
%The main result of this section
%is a version of the fractional (Sobolev--)Poincar\'e inequality.
%Due to the zero values on the set $E$, this kind of inequalities are often called
%boundary Poincar\'e inequalities. \av The proof uses completeness
%of $X$ via \cite[Lemma~5.1]{lehrbackHardyAssouad}. Hence, in the forthcoming applications 
%of Theorem \ref{t.assouad} we also  need completeness. \ed
%

The following theorem gives an upper bound for the upper Assouad codimension for
sets satisfying a capacity density condition.
We emphasize the strict inequality $\ucodima(E)<\beta p$, completeness and connectedness in the assumptions below.

\begin{theorem}\label{t.assouad_sufficient}
Assume that $X$ is a complete and connected  metric space. 
Let $1\le p<\infty$ and $0<\beta\le 1$. %and $\Lambda\ge 2$.
%and that there exists a constant $0<c_R<1$ such that 
%\jl the reverse doubling condition~\eqref{e.rev_dbl_decay} holds \ed
%\begin{equation}\label{e.decay_2}
%\mu(B(x,3r/4))\le c_R \mu(B(x,r))
%\end{equation}
%for each $x\in X$ and  every
%$0<r<\diam(X)/2$.
Let $E$ be a closed set with
$\ucodima(E)<\beta p$.
Then  $E$ satisfies a $(\beta,p)$-capacity density condition.
\end{theorem} 
%
%Note that \eqref{eq.bdry_poinc} is equivalent to the inequality
%\[
%\biggl(\intav_{B} \lvert u(x)\rvert^q\,d\mu(x)\biggr)^{p/q}
%\le 
%%Cr^{sp}\intav_{2\lambda B} \biggl(\int_{2\lambda B}\frac{|u(x)-u(y)|^t}{d(x,y)^{st}\mu(B(x,d(x,y)))}\,dy\biggr)^{p/t}\,d\mu(x)
%C \av R^{sp}\ed\intav_{\lambda B} 
%g_{u,s,t,\lambda B}(x)^p\,d\mu(x).
%%\biggl(\int_{2\lambda B}\frac{|u(x)-u(y)|^t}{d(x,y)^{st}\mu(B(x,d(x,y)))}\,dy\biggr)^{p}\,d\mu(x)
%\]
%
%
%
\begin{proof}
Fix $0<\eta<p$ such that 
$\ucodima(E)<\beta \eta$.
Let $x\in E$ and 
$0<r<\diam(E)/8$, and write $B=B(x,r)$.
By a simple covering argument using the doubling condition, it follows that
$\mathcal{H}^{\mu,\beta\eta}_{20r}(E\cap \iol{B})\ge 
C\mathcal{H}^{\mu,\beta\eta}_{r}(E\cap \iol{B})$ with a constant $C$ independent of $B$.
%Then $R<\diam(X)/2$.
Applying also Lemma~\ref{l.codim_sufficient} and then Lemma~\ref{l.measure_to_hausdorff}  gives
\begin{align*}
r^{\beta(p-\eta)}\cp_{\beta,p}(E\cap \iol{B},2B)
\ge C\mathcal{H}^{\mu,\beta\eta}_{20r}(E\cap \iol{B})\ge C\mathcal{H}^{\mu,\beta\eta}_{r}(E\cap \iol{B})
\ge r^{-\beta\eta}\mu(B)\,.
\end{align*}
After simplification, we obtain
\[
\cp_{\beta,p}(E\cap \iol{B},2B)\ge Cr^{-\beta p}\mu(B)\,,
\]
and the claim follows.
\end{proof}

Conversely, by using boundary Poincar\'e inequalities, it is 
easy to show that a capacity density condition
implies an upper bound for the upper Assouad codimension.

\begin{theorem}\label{t.easy_converse}
Let
$1\le p<\infty$ and $0<\beta\le 1$. Assume that $E\subset X$ satisfies a $(\beta,p)$-capacity density condition.
 Then
%Assume that there is a constant $0<c_R<1$ such that
%\begin{equation}\label{e.decay_2}
%\mu(B(x,r/2))\le c_R \mu(B(x,r))
%\end{equation}
%\color{blue} for each  $x\in X$ and  every
%$r>0$. \color{black}
%Let
%that $E\subset X$ is a (nonempty) closed and porous set such that
%%$\mathrm{diam}(E)=\infty$ and 
%\[
%\int_{B\setminus E} \frac{\lvert u(x)\rvert^{p-\varepsilon}}{d(x,E)^{\beta (p-\varepsilon)}}\,d\mu(x)
%\le C \int_{B}g(x)^p\big( M^{E,p}_{\beta} u(x)\big)^{-\varepsilon}\,d\mu(x)
%\]
%whenever $u\colon X\to\R$ is $\beta$-H\"older function  such that
%$u=0$ in $E$, $g\in\mathcal{D}_H^{\beta}(u)$, 
%and $B=B(w,r)$ with $w\in E$ and $0<r<\diam(E)$.
$\ucodima(E)\le \beta p$.
\end{theorem}

\begin{proof}
We adapt the proof of \cite[Theorem 5.3]{DLV2021} to our setting.
By using the doubling condition, 
it suffices to show that
\begin{equation}\label{eq.codim_est_easy}
\frac{\mu(E_r\cap B(w,R))}{\mu(B(w,R))}\ge c\Bigl(\frac{r}{R}\Bigr)^{\beta p},
\end{equation}
for all $w\in E$ and $0<r<R<\diam(E)/4$,
where the constant $c$ is independent of $w$, $r$ and $R$.

If $\mu(E_r\cap B(w,R))\ge \frac 1 2 \mu(B(w,R))$, the claim is clear since $\bigl(\frac{r}{R}\bigr)^{\beta p}\le 1$. Thus we may assume in the sequel that
$\mu(E_r\cap B(w,R)) < \tfrac 1 2 \mu(B(w,R))$, whence
\begin{equation}\label{eq.compl_meas_easy}
\mu(B(w,R)\setminus E_r) \ge \tfrac 1 2 \mu(B(w,R))>0.
\end{equation}

We define a  $\beta$-H\"older function $u\colon X\to\R$ by 
\[
u(x)=\min\{1,r^{-\beta}d(x,E)^\beta\}\,,\qquad  x\in X.
\]
Then $u=0$ in $E$, $u=1$ in $X\setminus E_{r}$, and 
\[
\lvert u(x)-u(y)\rvert \le r^{-\beta}d(x,y)^\beta \quad \text{ for all } x,y\in X.
\]
We obtain
\begin{equation}\label{eq.lhs_low_easy}
\begin{split}
R^{-\beta p}\int_{B(w,R)} \lvert u(x)\rvert^p\,d\mu(x) &\ge 
R^{-\beta p}\int_{B(w,R)\setminus E_{r}} \lvert u(x)\rvert^p\,d\mu(x)
\\&= 
R^{-\beta p}\mu(B(w,R)\setminus E_{r})
\ge \tfrac 1 2 R^{-\beta p}\mu(B(w,R)),
\end{split}
\end{equation}
where the last step follows from~\eqref{eq.compl_meas_easy}.

Since $u=1$ in $X\setminus E_{r}$ and
$u$ is a $\beta$-H\"older function with a constant $r^{-\beta}$, 
Lemma \ref{l.glueing} implies that
$g=r^{-\beta}\mathbf{1}_{E_{r}}\in\mathcal{D}_H^{\beta}(u)$.
We observe from \eqref{eq.lhs_low_easy}  and 
Theorem \ref{t.global_boundary_poincare} that
\begin{equation}\label{eq.rhs_upper_easy}
\begin{split}
C r^{-\beta p}\mu(E_r\cap B(w,R)) &= C\int_{B(w,R)}g(x)^p\,d\mu(x)\\
&\ge 
2 R^{-\beta p}\int_{B(w,R)} \lvert u(x)\rvert^p\,d\mu(x)\ge R^{-\beta p}\mu(B(w,R))\,,
\end{split}
\end{equation}
where the constant $C$ is independent of $w$, $r$ and $R$.
The claim~\eqref{eq.codim_est_easy} follows  from 
\eqref{eq.rhs_upper_easy}.
\end{proof}

%\begin{remark}
Observe that the upper bound $\ucodima(E)\le \beta p$ appears in the conclusion of Theorem~
\ref{t.easy_converse}. The rest of the paper is devoted  to  
showing the strict  inequality $\ucodima(E)< \beta p$  for $1<p<\infty$, which
 leads to a characterization of the $(\beta,p)$-capacity density condition in terms of this strict dimensional inequality.

 Our strategy is to combine the methods in \cite{MR3895752} and \cite{MR3673660}
 to prove a significantly stronger variant of the boundary Poincar\'e inequality, which involves maximal operators, see Theorem~\ref{t.main_boundedness}. We use this maximal inequality to prove a Hardy inequality, Theorem~\ref{t.improved}. 
This variant 
leads to 
%Theorem \ref{t.converse} and eventually to 
the characterization 
in Theorem~\ref{t.main_characterization} of the $(\beta,p)$-capacity density condition
in terms of 
the strict inequality $\ucodima(E)<\beta p$, among other geometric and analytic conditions. 
Certain additional geometric
assumptions are needed for the proof of Theorem~\ref{t.main_characterization},
namely geodesic property of $X$. 
We are not aware, to which extent this geometric assumption can be relaxed. 
%\end{remark}

\section{Local boundary Poincar\'e inequality}\label{s.truncation}

Our  next aim is to show Theorem~\ref{t.main_boundedness}, which concerns inequalities  localized to a fixed ball $B_0$ centered at $E$.  The  proof of this theorem requires that we first truncate the closed set $E$ to a smaller set $E_Q$ contained in a Whitney-type ball $\iol{Q}\subset B_0$ such that a local variant of the boundary Poincar\'e inequality remains valid. The choice of the Whitney-type ball $Q$ and the construction of the set $E_Q$ are given in this section. 

This truncation construction, that we borrow from \cite{MR3673660}, is done in such a way that a local Poincar\'e inequality holds, see Lemma~\ref{l.b_poincare}. This inequality is local in two senses: on one hand, the inequality holds only for balls $B\subset Q^*$; on the other hand, it holds for functions vanishing on the truncated set $E_Q$. Due to the subtlety of its consequences, the truncation in this section may seem arbitrary, but it is actually needed for our purposes.

 Assume that  $E$ is a closed set in a geodesic space $X$. Fix a ball $B_0=B(w,R)\subset X$ with $w\in E$ and $R<\diam(E)$. 
Define a family of balls
\begin{equation}\label{e.B_0}
\mathcal{B}_0=\{B\subset X\,:\, B\text{ is a ball such that } B\subset {B_0}\}\,.
\end{equation}
We also need a single Whitney-type ball  $Q=B(w,r_Q)\subset B_0$, where
\begin{equation}\label{e.q-radius}
 r_Q=\frac{R}{128}\,.
\end{equation}
The 4-dilation of the Whitney-type ball is denoted by $Q^*=4Q=B(w,4 r_Q)$.
 Now it holds that $Q^*\subsetneq X$, since otherwise 
\[
\diam(X)=\diam(Q^*)\le R/16< \diam(E)\le \diam(X)\,.
\] 
%Even though the Whitney balls need not be pairwise disjoint, they nevertheless have 
%the following standard covering properties with bounded overlap; cf.\ \cite[pp.~77--78]{MR2867756}.
%\begin{itemize}
%\item[(W1)] $B_0=\bigcup_{Q\in\mathcal{W}_0} Q$;
%\item[(W2)] $\sum_{Q\in\mathcal{W}_0} \mathbf{1}_{Q^*}\le C\mathbf{1}_{B_0}$ for some constant $C=C(c_\mu)>0$.
%\end{itemize}
The following properties (W1)--(W4) 
%for any Whitney ball $Q=B(x_Q,r_Q)\in\mathcal{W}_0$
are straightforward to verify.
For instance, property (W1) follows from inequality \eqref{e.diams}; we omit the simple proofs. 
%Below we refer to the family $\mathcal{B}_0$ of
%balls that is defined in \eqref{e.B_0} by using the fixed ball $B_0$.
\begin{itemize} 
\item[(W1)] If $B\subset X$ is a ball such that $B\cap \iol{Q}\not=\emptyset\not=2B\cap(X\setminus Q^*)$, then
$\diam(B)\ge 3r_Q/4$;
\item[(W2)] If  $B\subset Q^*$ is a ball, then $B\in\mathcal{B}_0$;
\item[(W3)] If $B\subset Q^*$ is a ball, $x\in B$  and $0<r\le \diam(B)$, then
$B(x,5r)\in\mathcal{B}_0$;
\item[(W4)] If $x\in Q^*$ and $0<r\le 2\diam(Q^*)$, then $B(x,r)\in\mathcal{B}_0$.
\end{itemize}
Observe that there is some overlap between
the properties (W2)--(W4). The slightly different formulations will conveniently 
guide the reader
in the sequel.

 The following Lemma \ref{l.truncation} gives us the truncated set $E_Q\subset \overline{Q}$ that contains big pieces of the original set $E$ at small scales. This big pieces property is not always satisfied
by $E\cap Q$, so it cannot be used instead.

\begin{lemma}\label{l.truncation}
Assume that $E\subset X$ is a closed set in a geodesic space $X$ and that $Q=B(w,r_Q)$ for $w\in E$ and $r_Q>0$.
Let $E_{Q}^0=E\cap \overline{\frac 1 2 Q}$, define inductively, for every $j\in\N$, that
\[
E_{Q}^{j}=\bigcup_{x\in E_{Q}^{j-1}} E\cap \overline{B(x,2^{-j-1}r)}\,,\quad
\text{ and set } \quad
E_{Q}=\overline{\bigcup_{j\in \N_0} E_{Q}^{j}}.
\]
Then the following statements hold:
\begin{itemize}
\item [(a)] $w\in E_{Q}$;
\item[(b)] $E_{Q}\subset E$;
\item[(c)] $E_{Q}\subset \overline{Q}$;
\item[(d)] $E^{j-1}_{Q}\subset E^{j}_{Q}\subset E_{Q}$  for every $j\in\N$.
\end{itemize}
\end{lemma}

%\begin{proof}
%Part (a) is is true since $w\in E_{B}^0$. Part (b) follows from the facts
%that $E$ is closed and $\cup_j E_{B}^j\subset E$ by definition.
%To verify (c), we fix $x\in E_{B}^j$.
%If $j=0$, then $x\in \overline{B}$. If $j>0$, then
%by induction we find a sequence $x_j, \ldots,x_0$ such $x_j=x$ and,
%for
%each $k=0,\ldots,j$,
%$x_k\in E^k_{B}$ and
%$x_k\in E\cap \overline{B(x_{k-1},2^{-k-1}r)}$ if $k>0$. It follows that
%\[
%d(x,w)\le \sum_{k=1}^j  d(x_k,x_{k-1}) + d(x_0,w) \le \sum_{k=1}^j 2^{-k-1}r + 2^{-1}r < r\,.
%\]
%Hence, $x\in B(w,r)\subset \overline{B}$.
%We have shown that $E_{B}^j\subset \overline{B}$ whenever $j\in\N_0$,
%whence it follows that also $E_{B}\subset \overline{B}$.
%To prove (d) we fix $j\in\N$ and $x\in E^{j-1}_{B}$. By definition
%we have $x\in E$ and, hence,
%$x\in E\cap B(x,2^{-j-1}r)\subset E^j_{B}$.
%\end{proof}

The next lemma shows that  the truncated set $E_Q$ in Lemma \ref{l.truncation} really
contains big pieces of the original set $E$ at all small scales. By using these balls
we later  employ the capacity density condition of $E$, see the proof of Lemma \ref{l.b_poincare} for details.

\begin{lemma}\label{l.pallot}
Let $E$, $Q$, and $E_{Q}$ be as in Lemma~\ref{l.truncation}.
Suppose that  $m\in\N_0$ and $x\in X$ is such that
$d(x,E_{Q}) <2^{-m+1}r_Q$.
Then there exists
a ball $\widehat{B}=B(y_{x,m},2^{-m-1}r_Q)$ such that $y_{x,m}\in E$,
\begin{equation}\label{e.bubble} 
E\cap \overline{2^{-1}\widehat{B}}=E_Q\cap \overline{2^{-1}\widehat{B}}\,,
\end{equation}
and $\widehat{B}\subset B(x,2^{-m+2}r_Q)$.
\end{lemma}

We refer to \cite{MR3673660} for the proofs of Lemma \ref{l.truncation} and Lemma \ref{l.pallot}.
A similar truncation procedure is a standard
technique when proving the self-improvement of different capacity density conditions. It originally appears in~\cite[p. 180]{MR946438} for Riesz capacities in $\R^n$,
and later also in~\cite{MR1386213} for $\R^n$
and in~\cite{MR1869615} for general metric spaces.

With the aid of big pieces inside the truncated set $E_Q$, we can show that a localized variant of the boundary Poincar\'e inequality
in Theorem \ref{t.global_boundary_poincare}
holds for the truncated set $E_Q$, if $E$ satisfies
a capacity density condition.

\begin{lemma}\label{l.b_poincare}
Let $X$ be a geodesic space. 
Assume that $1\le p<\infty$ and $0<\beta\le 1$. Suppose that
a closed set $E\subset X$ satisfies the $(\beta,p)$-capacity density condition with a constant $c_0$.
Let $B_0=B(w,R)\subset X$ be a ball with $w\in E$ and $R<\diam(E)$,  and let $Q=B(w,r_Q)\subset B_0$ be the corresponding Whitney-type ball. Assume that
$B\subset Q^*$ is a ball with a center $x_B\in E_Q$. Then
there is a constant $K= K(p,c_\mu,c_0)$ such that
\begin{equation}\label{e.local_boundary_poincare}
\vint_{B} \lvert u(x)\rvert^p\,d\mu(x)\le K\diam(B)^{\beta p}\vint_B g(x)^p\,d\mu(x)
\end{equation}
for all $\beta$-H\"older functions $u$ in $X$ with $u=0$ in $E_Q$, and
for all $g\in\mathcal{D}_H^{\beta}(u)$. 
\end{lemma}

\begin{proof}
Fix a ball $B=B(x_B,r_B)\subset Q^*$ with $x_B\in E_Q$.
Recall that
 $r_Q=R/128$ as in \eqref{e.q-radius}.
 Since $B\subset Q^*\subsetneq X$, we have
\[
0<r_B\le \diam(B)\le \diam(Q^*)\le 8r_Q\,.
\] 
Hence, we can choose $m\in\N_0$ such that
$2^{-m+2}r_Q<r_B\le 2^{-m+3}r_Q$. Then
\[
d(x_B,E_Q)=0 < 2^{-m+1}r_Q\,.
\]
By Lemma \ref{l.pallot} with  $x=x_B$  there exists a ball  $\widehat{B}=B(y,2^{-m-1}r_Q)$ 
such that  $y\in E$, 
\begin{equation}\label{e.trunc}
E\cap \overline{2^{-1}\widehat{B}}=E_Q\cap \overline{2^{-1}\widehat{B}}
\end{equation}
and $\widehat{B}\subset B(x_B,2^{-m+2}r_Q)\subset B(x_B,r_B)=B$.  Observe also
that $B\subset 32\widehat{B}$. 

 Fix a $\beta$-H\"older function $u$ in $X$ with $u=0$ in $E_Q$, and let
 $g\in\mathcal{D}_H^{\beta}(u)$.
We estimate
\begin{align*}
\vint_{B} \lvert u(x)\rvert^p\,d\mu(x)
\le C(p)\vint_{B} \lvert u(x)-u_{B}\rvert^p\,d\mu(x) + C(p)\lvert u_{B}-u_{\widehat{B}}\rvert^p+C(p)\lvert u_{\widehat{B}}\rvert^p\,.
\end{align*}
By the $(\beta,p,p)$-Poincar\'e inequality in Theorem \ref{t.pp_poincare},  we obtain
\[
\vint_{B} \lvert u(x)-u_{B}\rvert^p\,d\mu(x)\le C(p)\diam(B)^{\beta p}\vint_B g(x)^p\,d\mu(x)\,.
\]
Using also H\"older's inequality and the doubling condition, we get
\begin{align*}
\lvert u_{B}-u_{\widehat{B}}\rvert^p&\le \vint_{\widehat{B}} \lvert u(x)-u_{B}\rvert^p\,d\mu(x)
\\&\le C(c_\mu)\vint_{B} \lvert u(x)-u_{B}\rvert^p\,d\mu(x)
\le C(p,c_\mu)\diam(B)^{\beta p}\vint_{B} g(x)^p\, d\mu(x)\,.
\end{align*}
 In order to estimate
 the remaining term $\lvert u_{\widehat{B}}\rvert^p$, we write
$ \{u=0\}= \{y\in X :  u(y)=0\}\supset E_Q$.
By using the monotonicity of capacity,
identity \eqref{e.trunc}, the assumed capacity density condition, and the doubling condition, we obtain
\begin{align*}
{\cp_{\beta,p}(\{u=0\}\cap \overline{2^{-1}\widehat{B}},\widehat{B})}
&\ge {\cp_{\beta,p}(E_Q\cap \overline{2^{-1}\widehat{B}},\widehat{B})}
={\cp_{\beta,p}(E\cap \overline{2^{-1}\widehat{B}},\widehat{B})}
\\&\ge c_0 (2^{-m-2}r_Q)^{-\beta p}\mu(2^{-1}\widehat{B})\ge C(c_\mu,c_0) r_B^{-\beta p} \mu(B)\,.
\end{align*}
By Theorem~\ref{t.Mazya_Poincare}, we obtain
\begin{align*}
\lvert u_{\widehat{B}}\rvert^p&\le \vint_{\widehat{B}} \lvert u(x)\rvert^{p}\,d\mu(x)
\le C(p)\left({\cp_{\beta,p}(\{u=0\}\cap \overline{2^{-1}\widehat{B}},\widehat{B})}\right)^{-1}\int_{\widehat{B}} g(x)^p\,d\mu(x)\\
&\le C(p,c_\mu,c_0)\frac{r_B^{\beta p}}{\mu(B)} \int_{\widehat{B}} g(x)^p\,d\mu(x)\le C(p,c_\mu,c_0)\diam(B)^{\beta p} \vint_{B} g(x)^p\,d\mu(x)\,.
\end{align*}
The proof is completed by combining the above estimates for the three terms.
\end{proof}

\section{ Maximal boundary Poincar\'e inequalities}\label{s.main}

We formulate and prove our key results, Theorem~\ref{t.main_boundedness} and Theorem~\ref{t.main_local}. These theorems give 
improved variants of the local boundary Poincar\'e inequality \eqref{e.local_boundary_poincare}. The improved variants
are 
norm inequalities for a combination of two maximal functions.
Hence, we can view Theorem~\ref{t.main_boundedness} and Theorem~\ref{t.main_local} as maximal boundary Poincar\'e inequalities.
Our treatment adapts \cite{MR3895752} to the setting of boundary Poincar\'e inequalities.
%, to which we include nontrivial modifications in the argument. We provide the somewhat lengthy and technical details.

\begin{definition}
Let $X$ be a geodesic space, $1<p<\infty$ and $0<\beta\le 1$.
%Let $B_0=B(w,R)$ be a ball in $X$ such that $w\in E$ and $0<R<\diam(E)$. We also let
%$Q$ and $E_Q$ be as in Section \ref{s.truncation} with respect to $B_0$.
If $\mathcal{B}\not=\emptyset$ is a given family
of balls in $X$, then we define a fractional sharp maximal function 
\begin{equation}\label{d.m_def}
M^{\sharp,p}_{\beta,\mathcal{B}}u(x)=\sup_{x\in B\in \mathcal{B}} \bigg(\frac{1}{\diam(B)^{\beta p}}\vint_B \lvert u(y)-u_B\rvert^p\,d\mu(y)\bigg)^{1/p}\,,\qquad x\in X\,,
\end{equation}
whenever $u\colon X\to \R$ is a $\beta$-H\"older function. 
We also define the maximal function adapted to a given set $E_Q\subset X$ by
\begin{equation}\label{d.m_e_def}
M^{E_Q,p}_{\beta,\mathcal{B}}u(x)=\sup_{x\in B\in \mathcal{B}} \bigg(\frac{\mathbf{1}_{E_Q}(x_B)}{\diam(B)^{\beta p}}\vint_B \lvert u(y)\rvert^p\,d\mu(y)\bigg)^{1/p}\,,\qquad x\in X\,,
\end{equation} 
whenever $u\colon X\to \R$ is a $\beta$-H\"older function such that $u=0$ in $E_Q$.
Here $x_B$ is the center of the ball $B\in\mathcal{B}$. 
The supremums in \eqref{d.m_def} and \eqref{d.m_e_def} are defined to be zero, if there is no ball $B$ in $\mathcal{B}$ that contains the point $x$.
\end{definition}

We are mostly interested in 
maximal functions for the ball family
\eqref{e.B_0}. The following is the main result in this section.

\begin{theorem}\label{t.main_boundedness}
Let $X$ be a geodesic space. Let
$1<p<\infty$ and $0<\beta\le 1$. Let $E\subset X$ be a closed
set which satisfies the $(\beta,p)$-capacity density condition with a constant $c_0$.
Let $B_0=B(w,R)$ be a ball with $w\in E$ and $R<\diam(E)$.
Let  $E_Q$ be the truncation
of $E$ to the Whitney-type ball $Q$ as in Section \ref{s.truncation}.
Then
there exists 
%$0<\varepsilon<p-1$ and 
a constant $C=C(\beta,p,c_\mu,c_0)>0$ such that
 inequality
\begin{equation}\label{e.loc_des_C}
\begin{split}
\int_{B_0}\big(M^{\sharp,p}_{\beta,\mathcal{B}_0}u+M^{E_Q,p}_{\beta,\mathcal{B}_0}u \big)^{p}\,d\mu
\le 
C\int_{B_0} g^{p}\,d\mu
\end{split}
\end{equation}
holds whenever $u\in\mathrm{Lip}_\beta(X)$ is such that $u=0$ in $E_Q$ and $g\in\mathcal{D}_H^{\beta}(u)$.
\end{theorem}

\begin{proof}
We use the following Theorem \ref{t.main_local} with $\varepsilon=0$.
Observe that the first term on the right-hand side of~\eqref{e.loc_des}
is finite, since $u$ is  a $\beta$-H\"older function in $X$ such that $u=0$ in $E_Q$. Inequality~\eqref{e.loc_des_C} is obtained when this term is absorbed to the left-hand side after choosing the number
 $k$ large enough, depending only on $\beta$, $p$, $c_\mu$ and $c_0$.
\end{proof}

\begin{theorem}\label{t.main_local}
Let $X$ be a geodesic space. Let 
$1<q<p<\infty$ and $0<\beta\le 1$  be such that the
$(\beta,p,q)$-Poincar\'e inequality in Theorem \ref{t.qp_poincare} holds.  Let $E\subset X$ be a closed
set satisfying the $(\beta,p)$-capacity density condition with a constant $c_0$.
Let $B_0=B(w,R)$ be a ball with $w\in E$ and $R<\diam(E)$.
Let  $E_Q$ be the truncation
of $E$ to the Whitney-type ball $Q=B(w,r_Q)\subset B_0$ as in Section \ref{s.truncation}.
Let $K=K(p,c_\mu,c_0)>0$ be the constant for the local boundary Poincar\'e inequality 
in Lemma~\ref{l.b_poincare}.
Assume that $k\in \N$,  $0\le \varepsilon< (p-q)/2$,
 and $\alpha=\beta p^2/(2(s+\beta p))>0$ with $s=\log_2 c_\mu$. 
Then inequality
\begin{equation}\label{e.loc_des}
\begin{split}
&\int_{B_0}\big(M^{\sharp,p}_{\beta,\mathcal{B}_0}u+M^{E_Q,p}_{\beta,\mathcal{B}_0}u \big)^{p-\varepsilon}\,d\mu
\le  C_1\ed \bigg(2^{k(\varepsilon-\alpha)}+\frac{K 4^{k\varepsilon}}{k^{p-1}}\bigg)\int_{B_0} \big( M^{\sharp,p}_{\beta,\mathcal{B}_0}u+M^{E_Q,p}_{\beta,\mathcal{B}_0}u \big)^{p-\varepsilon}\,d\mu
\\&\quad+ C_1 \ed C(k,\varepsilon) K \int_{B_0\setminus \{M^{\sharp,p}_{\beta,\mathcal{B}_0}u+M^{E_Q,p}_{\beta,\mathcal{B}_0}u =0\}} g^p\big(M^{\sharp,p}_{\beta,\mathcal{B}_0}u+M^{E_Q,p}_{\beta,\mathcal{B}_0}u \big)^{-\varepsilon}\,d\mu + C_3\ed \int_{B_0} g^{p-\varepsilon}\,d\mu\ed
\end{split}
\end{equation}
holds for each $u\in\mathrm{Lip}_\beta(X)$ 
with $u=0$ in $E_Q$ 
and every $g\in\mathcal{D}_H^{\beta}(u)$. Here $C_1=C_1(\beta,p,c_\mu)$, $C_1=C_1(\beta,p,c_\mu)$,
$ C_3=C(\beta,p,c_\mu)\ed$, $C(k,\varepsilon)=(4^{k\varepsilon}-1)/\varepsilon$ if 
$\varepsilon>0$ and $C(k,0)=k$.
\end{theorem}

 \begin{remark}
Observe that Theorem \ref{t.main_local} implies a variant of Theorem \ref{t.main_boundedness}
when we choose $\varepsilon>0$ to be sufficiently small. We omit the formulation
of this variant, since we do not use it. This is because of the following defect:
one of the terms is the integral of 
$g^p\big(M^{\sharp,p}_{\beta,\mathcal{B}_0}u+M^{E_Q,p}_{\beta,\mathcal{B}_0}u \big)^{-\varepsilon}$
instead of $g^{p-\varepsilon}$.
Because of its independent interest,
we have however chosen to formulate Theorem~\ref{t.main_local} such that
it  incorporates the parameter
$\varepsilon$.
\end{remark}

The proof of Theorem \ref{t.main_local} is completed in Section \ref{ss.main_local}. 
For the proof, we need preparations
that are treated in Sections \ref{ss.fs} -- \ref{ss.auxiliary_local}.
At this stage, we already fix $X$, $E$, $B_0$, $Q$, $E_Q$, $K$,  $\mathcal{B}_0$, $p$, $\beta$, $q$, $\varepsilon$, $k$ and $u$ as
in the statement of Theorem \ref{t.main_local}.
Notice, however, that the $\beta$-Haj{\l}asz gradient  $g$ of $u$ is not yet fixed.  We abbreviate
$M^{\sharp} u=M^{\sharp,p}_{\beta,\mathcal{B}_0}u$ and $M^{E_Q}u=M^{E_Q,p}_{\beta,\mathcal{B}_0}u$, 
and denote \[U^{\lambda}=\{x\in B_0\,:\, M^{\sharp} u(x)+M^{E_Q}u(x) >\lambda\}\,,\qquad \lambda>0\,.\]
The sets $U^\lambda$ are open in $X$. 
If $F\subset X$ is a Borel set and $\lambda>0$, we write $U^\lambda_F=U^\lambda \cap F$. 
We refer to these objects throughout 
Section \ref{s.main}
without further notice.    

\subsection{Localization to Whitney-type ball}\label{ss.fs}

We need a  smaller maximal function that is localized to the Whitney-type ball $Q$.
Consider the ball family
%\footnote{Let us emphasize that it is important to use $Q^*$ in the 
%definition for $\mathcal{B}_Q$ instead of $Q$.
%}
\[
\mathcal{B}_{Q}=\{B\subset X\,:\,  B  \text{ is a ball such that }B\subset Q^*\}
\]
and define 
\begin{equation}\label{e.m_loc_q}
M^{E_Q}_{Q} u=\mathbf{1}_{Q^*}M^{E_Q,p}_{\beta,\mathcal{B}_{Q}} u\,.
\end{equation}
%By using these individual maximal functions, we then define a {\em Whitney-ball localized sharp maximal function}\footnote{It is
%equally important to use $\mathbf{1}_Q$ instead of $\mathbf{1}_{Q^*}$ in the
%definition of $M^\sharp_{\mathrm{loc}} u$; these are delicate matters and related to the latter selection of stopping balls with the aid of condition (W1).}
%\[
%M^\sharp_{\textup{loc}} u = \sup_{Q\in\mathcal{W}_0} \mathbf{1}_Q M^\sharp_{Q} u\,.
%\]
If $\lambda>0$, we write \begin{equation}\label{e.super}
Q^\lambda = \{x\in Q^*\,:\, M^{E_Q}_{Q} u(x)>\lambda\}\,.
\end{equation}
%We will later cover $Q^\lambda$  by $5$-dilations of balls
%coming from a stopping time argument in \S\ref{s.stopping}. 
%These stopping balls $B$
%will be contained in $Q^*$ and thus Lemma \ref{l.arm_local} applies, showing
%that $u|(E_Q\cup (B\setminus U^\lambda))$ is $\beta$-H\"older with a constant  $\kappa=C(\beta,c_\mu)\lambda$.
%This is critical for the truncation argument in the proof of Lemma~\ref{l.mainl_local}.
We estimate 
the left-hand side of \eqref{e.loc_des_C} 
in terms of \eqref{e.m_loc_q} with the aid of the  following norm estimate.
We will later be able to estimate the smaller maximal function \eqref{e.m_loc_q}.

\begin{lemma}\label{l.big_t_o_small_ball}
There are constants $C_1=C(p,c_\mu)$ and $C_2=C(\beta,p,c_\mu)$
such that 
\begin{align*}
&\int_{B_0} \big(M^{\sharp} u(x)+M^{E_Q}u(x)\big)^{p-\varepsilon}\,d\mu(x)\\
&\qquad \le C_1\int_{B_0} \big(M^{E_Q}_{Q} u(x)\big)^{p-\varepsilon}\,d\mu(x) + C_2 \int_{B_0} g(x)^{p-\varepsilon}\,d\mu(x)
\end{align*}
for all $g\in\mathcal{D}_H^{\beta}(u)$. 
\end{lemma}

\begin{proof}
Fix $g\in\mathcal{D}_H^{\beta}(u)$.
We have
\begin{equation}\label{e.kaksi}
\begin{split}
&\int_{B_0} \big(M^{\sharp} u(x)+M^{E_Q}u(x)\big)^{p-\varepsilon}\,d\mu(x)\\
&\qquad \le C(p)\int_{B_0}\left(M^{\sharp} u(x)\right)^{p-\varepsilon}d\mu(x) + C(p)\int_{B_0}\left(M^{E_Q} u(x)\right)^{p-\varepsilon}d\mu(x)\,.
\end{split}
\end{equation}
Let $x\in B_0$ and let $B\in\mathcal{B}_0$ be such that
$x\in B$.
By \eqref{e.B_0} and the $(\beta,p,q)$-Poincar\'e inequality, see Theorem~\ref{t.qp_poincare},  we obtain 
\begin{align*}
&\bigg(\frac{1}{\diam(B)^{\beta p}}\vint_B \lvert u(y)-u_B\rvert^p\,d\mu(y)\bigg)^{1/p}
\\&\qquad \le C(c_\mu,p,\beta)\bigg(\vint_B g(y)^q\,d\mu(y)\bigg)^{1/q}
\le C(c_\mu,p,\beta)(M(g^q\mathbf{1}_{B_0})(x))^{\frac{1}{q}}\,.
\end{align*}
Here $M$ is the non-centered Hardy--Littlewood maximal function operator.
By taking supremum over balls $B$ as above, we obtain
\[
M^{\sharp} u(x)=M^{\sharp,p}_{\beta,\mathcal{B}_0}u(x)\le
 C(\beta,p,c_\mu)\ed (M(g^q\mathbf{1}_{B_0})(x))^{\frac{1}{q}}\,.
\]
Since $p-\varepsilon>q$, the Hardy--Littlewood maximal
function theorem \cite[Theorem~3.13]{MR2867756} implies that
\begin{align*}
\int_{B_0}\left(M^{\sharp} u(x)\right)^{p-\varepsilon}d\mu(x)
&\le C(\beta,p,c_\mu)\int_{B_0}(M(g^q\mathbf{1}_{B_0})(x))^{\frac{p-\varepsilon}{q}}d\mu(x)
\\&\le \frac{C(\beta,p,c_\mu)}{p-q-\varepsilon}\ed \int_{B_0} g(x)^{p-\varepsilon}\,d\mu(x)\,.
\end{align*}
Since $\varepsilon<(p-q)/2$, 
this provides an estimate for the first term in the right-hand side of \eqref{e.kaksi}.

In order to estimate the second term in the right-hand side of \eqref{e.kaksi}, we let
$x\in B_0\setminus Q^*$ and let $B\in\mathcal{B}_0$ be such that
$x\in B$. We will estimate the term
\[
\bigg(\frac{\mathbf{1}_{E_Q}(x_B)}{\diam(B)^{\beta p}}\vint_B \lvert u(y)\rvert^p\,d\mu(y)\bigg)^{1/p}\,,
\]
where $x_B$ is the center of $B$.
Clearly we may assume that $x_B\in E_Q\subset \iol{Q}$. By condition (W1),  we see
that $\mathrm{diam}(B)\ge C\diam(B_0)$ and $\mu(B)\ge C(c_\mu)\mu(B_0)$. Since
$B\in\mathcal{B}_0$, we have $B\subset B_0$. Thus,
\[
\bigg(\frac{\mathbf{1}_{E_Q}(x_B)}{\diam(B)^{\beta p}}\vint_B \lvert u(y)\rvert^p\,d\mu(y)\bigg)^{1/p}\le  C(p,c_\mu)
\bigg(\frac{1}{\diam(B_0)^{\beta p}}\vint_{B_0} \lvert u(y)\rvert^p\,d\mu(y)\bigg)^{1/p}\,.
\]
By taking supremum over balls $B$ as above, we obtain 
\[
M^{E_Q}u(x)=M^{E_Q,p}_{\beta,\mathcal{B}_0}u(x)\le C(p,c_\mu)\bigg(\frac{1}{\diam(B_0)^{\beta p}}\vint_{B_0} \lvert u(y)\rvert^p\,d\mu(y)\bigg)^{1/p}
\]
for all $x\in B_0\setminus Q^*$.
By integrating, we obtain
\begin{equation}\label{e.out_ball}
\begin{split}
&\int_{B_0\setminus Q^*}\left(M^{E_Q} u(x)\right)^{p-\varepsilon}d\mu(x)\\&\qquad \le 
C(p,c_\mu)\mu(B_0)\bigg(\frac{1}{\diam(B_0)^{\beta p}}\vint_{B_0} \lvert u(y)\rvert^p\,d\mu(y)\bigg)^{\frac{p-\varepsilon}{p}}\\
&\qquad \le \frac{C(p,c_\mu)\mu(B_0)}{\diam(B_0)^{\beta(p-\varepsilon)}}\left[ \left(  
\vint_{B_0} \lvert u(y)-u_{Q^*}\rvert^p\,d\mu(y)\right)^\frac{p-\varepsilon}{p} + \lvert u_{Q^*}\rvert^{p-\varepsilon}\right]\,.
\end{split}
\end{equation}
By the $(\beta,p,q)$-Poincar\'e inequality and H\"older's inequality with $q<p-\varepsilon$, we obtain
\begin{align*}
&\frac{C(p,c_\mu)\mu(B_0)}{\diam(B_0)^{\beta(p-\varepsilon)}}\left(  
\vint_{B_0} \lvert u(y)-u_{Q^*}\rvert^p\,d\mu(y)\right)^\frac{p-\varepsilon}{p}
\\&\qquad \le  \frac{C(p,c_\mu)\mu(B_0)}{\diam(B_0)^{\beta(p-\varepsilon)}}\left[\left(  
\vint_{B_0} \lvert u(y)-u_{B_0}\rvert^p\,d\mu(y)\right)^\frac{p-\varepsilon}{p}
+ \lvert u_{B_0}-u_{Q^*}\rvert^{p-\varepsilon}\right]\\
&\qquad \le \frac{C(p,c_\mu)\mu(B_0)}{\diam(B_0)^{\beta(p-\varepsilon)}}\left(\vint_{B_0} \lvert u(y)-u_{B_0}\rvert^p\,d\mu(y)\right)^\frac{p-\varepsilon}{p}\\
&\qquad \le C(\beta,p,c_\mu)\mu(B_0)\bigg(\vint_{B_0} g(x)^q\,d\mu(x)\bigg)^{\frac{p-\varepsilon}{q}}\\
&\qquad\le  C(\beta,p,c_\mu)\int_{B_0} g(x)^{p-\varepsilon}\,d\mu(x)\,.
\end{align*}
On the other hand, since $Q^*=B(w,4r_Q)$ with $w\in E_Q$ and $r_Q=R/128$, \ed we have
\begin{align*}
\frac{C(p,c_\mu)\mu(B_0)}{\diam(B_0)^{\beta(p-\varepsilon)}} \lvert u_{Q^*}\rvert^{p-\varepsilon}
&\le C(p,c_\mu)\frac{\mu(Q^*)}{\diam(Q^*)^{\beta(p-\varepsilon)}}\lvert u_{Q^*}\rvert^{p-\varepsilon}
\\& \le C(p,c_\mu)\int_{Q^*} \left(
\frac{\mathbf{1}_{E_Q}({w})}{\diam(Q^*)^{\beta p}}\vint_{Q^*} \lvert u(y)\rvert^p\,d\mu(y) \right)^{\frac{p-\varepsilon}{p}}\,d\mu(x)\\
&\le C(p,c_\mu)\int_{Q^*} \big(\mathbf{1}_{Q^*}(x) M^{E_Q,p}_{\beta,\mathcal{B}_{Q}} u(x)\big)^{p-\varepsilon}\,d\mu(x)\\
&= C(p,c_\mu)\int_{B_0} \big(M^{E_Q}_{Q} u(x)\big)^{p-\varepsilon}\,d\mu(x)\,.
\end{align*}
This concludes the estimate for the integral in \eqref{e.out_ball} over
$B_0\setminus Q^*$.

To estimate the integral over the set $Q^*$, we fix $x\in Q^*$. Let
$B\in\mathcal{B}_0$ be such that $x\in B$. If $B\subset Q^*$, then
\begin{align*}
\bigg(\frac{\mathbf{1}_{E_Q}(x_B)}{\diam(B)^{\beta p}}\vint_B \lvert u(y)\rvert^p\,d\mu(y)\bigg)^{1/p}\le \mathbf{1}_{Q^*}(x) M^{E_Q,p}_{\beta,\mathcal{B}_{Q}} u(x)=M^{E_Q}_{Q} u(x)\,.
\end{align*}
Next we consider the case $B\not\subset Q^*$, and again we need to estimate the quantity
\[
\bigg(\frac{\mathbf{1}_{E_Q}(x_B)}{\diam(B)^{\beta p}}\vint_B \lvert u(y)\rvert^p\,d\mu(y)\bigg)^{1/p}\,.
\]
We may assume that $x_B\in E_Q\subset \iol{Q}$. By condition (W1),  we obtain $\diam(B)\ge C\diam(B_0)$ and $\mu(B)\ge C(c_\mu)\mu(B_0)$.
Hence,
\[
\bigg(\frac{\mathbf{1}_{E_Q}(x_B)}{\diam(B)^{\beta p}}\vint_B \lvert u(y)\rvert^p\,d\mu(y)\bigg)^{1/p}\le  C(p,c_\mu)
\bigg(\frac{1}{\diam(B_0)^{\beta p}}\vint_{B_0} \lvert u(y)\rvert^p\,d\mu(y)\bigg)^{1/p}\,.
\]
By taking supremum over balls $B$ as above, we obtain 
\[
M^{E_Q}u(x)\le M^{E_Q}_{Q} u(x)+C(p,c_\mu)
\bigg(\frac{1}{\diam(B_0)^{\beta p}}\vint_{B_0} \lvert u(y)\rvert^p\,d\mu(y)\bigg)^{1/p}
\]
for all $x\in Q^*$.
It follows that
\begin{align*}
&\int_{Q^*} \left( M^{E_Q}u(x)\right)^{p-\varepsilon}\,d\mu(x)\\
&\qquad \le 
C(p,c_\mu)\int_{B_0} \big(M^{E_Q}_{Q} u(x)\big)^{p-\varepsilon}\,d\mu(x) + 
C(p,c_\mu)\mu(B_0)\bigg(\frac{1}{\diam(B_0)^{\beta p}}\vint_{B_0} \lvert u(y)\rvert^p\,d\mu(y)\bigg)^{\frac{p-\varepsilon}{p}}\,.
\end{align*}
We can now estimate as above, and complete the proof.
\end{proof}

The following lemma is variant of \cite[Lemma 4.12]{MR3895752}.
We also refer to \cite[Lemma 3.6]{MR1681586}.

\begin{lemma}\label{l.arm_local}
Fix $x,y\in Q^*$. Then
\begin{equation}\label{e.des_hld}
\lvert u(x)-u(y)\rvert \le C(\beta,c_\mu)\:  d(x,y)^\beta 
\big(M^{\sharp} u(x)+M^{\sharp} u(y)\big)
\end{equation}
and
\begin{equation}\label{e.des_sec}
\lvert u(x)\rvert \le 
C(\beta,c_\mu) \: d(x,E_Q)^\beta \left(M^{\sharp} u(x)+M^{E_Q} u(x)\right)\,.
\end{equation}
Furthermore, assuming that  $\lambda>0$, 
then the restriction 
$u|_{E_Q\cup (Q^*\setminus U^\lambda)}$ is a $\beta$-H\"older function in the set $E_Q\cup (Q^*\setminus U^\lambda)$ 
with constant $\kappa=C(\beta,c_\mu)\lambda$.
\end{lemma}

\begin{proof}
The property (W4) is used below several times without further notice. \ed
Let $z\in Q^*$ and $0<r\le 2\diam(Q^*)$. Write
$B_i=B(z,2^{-i}r)\in\mathcal{B}_0$ for each $i\in \{0,1,\ldots\}$. Then, with the standard `telescoping' argument, see for instance the proof of \cite[Lemma 3.6]{MR1681586}, we obtain
\begin{align*}
\lvert u(z)-u_{B(z,r)}\rvert
&\le c_\mu \sum_{i=0}^\infty \vint_{B_i} \lvert u-u_{B_i}\rvert\,d\mu \\ 
&\le c_\mu \sum_{i=0}^\infty 2^{\beta(1-i)}r^\beta \bigg(\frac{1}{\diam(B_i)^{\beta p}}\vint_{B_i} \lvert u-u_{B_i}\rvert^p\,d\mu\bigg)^{1/p}\\
%:
&\le c_\mu  M^{\sharp}u(z) \sum_{i=0}^\infty 2^{\beta(1-i)}r^\beta 
\le C(\beta,c_\mu)\, r^\beta M^{\sharp}u(z)\,.
\end{align*}
Fix $x,y\in Q^*$.
Since $0<d=d(x,y)\le \diam(Q^*)$, we obtain
\begin{align*}
\lvert u(y)-u_{B(x,d)}\rvert &\le \lvert u(y)-u_{B(y,2d)}\rvert+\lvert u_{B(y,2d)}-u_{B(x,d)}\rvert\\
&\le C(\beta,c_\mu)\, d^\beta M^{\sharp}u(y) + \frac{\mu(B(y,2d))}{\mu(B(x,d))}\vint_{B(y,2d)} \lvert u-u_{B(y,2d)}\rvert\,d\mu\\
&\le C(\beta,c_\mu)\,d^\beta \Biggl[M^{\sharp}u(y) +  \Biggl(\frac{1}{\diam(B(y,2d))^{\beta p}}\vint_{B(y,2d)} \lvert u-u_{B(y,2d)}\rvert^p\,d\mu\bigg)^{1/p}\Biggr]\\
&\le C(\beta,c_\mu)\, d^\beta M^{\sharp}u(y)\,.
\end{align*}
It follows that
\begin{align*}
\lvert u(x)-u(y)\rvert&\le \lvert u(x)-u_{B(x,d)}\rvert + \lvert u_{B(x,d)}-u(y)\rvert
\le C(\beta,c_\mu)\, d(x,y)^\beta  \big( M^{\sharp}u(x)+ M^{\sharp}u(y) \big)\,,
\end{align*}
which is the desired inequality \eqref{e.des_hld}.

To prove inequality \eqref{e.des_sec}, we let $x\in Q^*$. If
$d(x,E_Q)=0$, then $x\in E_Q$ and we are done since $u=0$ in $E_Q$. Therefore we may
assume that $d(x,E_Q)>0$. Then there exists
$y\in E_Q\subset \iol{Q}\subset Q^*$ such that $d=d(x,y)<\min\{2d(x,E_Q), \diam(Q^*)\}$ and we have
\begin{align*}
\lvert u(x)\rvert &\le \lvert u(x)-u_{B(y,d)}\rvert + \lvert u_{B(y,d)}\rvert
\\ &\le C(\beta,c_\mu)d^\beta M^{\sharp}u(x)
+ c_\mu\vint_{B(y,2d)} \lvert u\rvert\,d\mu\\
&\le C (\beta,c_\mu)d^\beta M^{\sharp}u(x)
+
c_\mu(4d)^\beta\left(\frac{\mathbf{1}_{E_Q}(y)}{\diam(B(y,2d))^{\beta p}}\vint_{B(y,2d)}\lvert u\rvert^p \,d\mu\right)^{\frac{1}{p}}\\&\le 
C(\beta,c_\mu)d^\beta \left(M^{\sharp} u(x)+M^{E_Q} u(x)\right)\\
&\le C(\beta,c_\mu)d(x,E_Q)^\beta \left(M^{\sharp} u(x)+M^{E_Q} u(x)\right)\,.
\end{align*}
Inequality \eqref{e.des_sec} follows.

Fix $\lambda>0$. Next we show that $u|(E_Q\cup (Q^*\setminus U^\lambda))$ is $\beta$-H\"older
with constant $\kappa=C(\beta,c_\mu)\lambda$.
Let $x,y\in E_Q\cup (Q^*\setminus U^{\lambda})$. There are four cases
to be considered. First, if $x,y\in E_Q$, then
\[
\lvert u(x)-u(y)\rvert=0\le \kappa d(x,y)^\beta,
\]
since $u=0$  in $E_Q$. If $x,y\in Q^*\setminus U^\lambda$, then
we apply \eqref{e.des_hld} and obtain
\[
\lvert u(x)-u(y)\rvert\le C(\beta,c_\mu) d(x,y)^\beta
\big(M^{\sharp} u(x)+M^{\sharp} u(y)\big)\le C(\beta,c_\mu) \lambda d(x,y)^\beta\,.
\]
Here we also used the fact that $Q^*\subset B_0$. If $x\in E_Q$ and
$y\in Q^*\setminus U^\lambda$, we apply \eqref{e.des_sec} and get 
\begin{align*}
\lvert u(x)-u(y)\rvert = \lvert u(y)\rvert \le C(\beta,c_\mu)d(y,E_Q)^\beta \left(M^{\sharp} u(y)+M^{E_Q} u(y)\right)\le C(\beta,c_\mu)\lambda d(x,y)^\beta\,.
\end{align*}
The last case $x\in Q^*\setminus U^\lambda$ and $y\in E_Q$  is 
treated in similar way.
\end{proof}

\subsection{Stopping construction}\label{s.stopping}

 We  continue as in \cite{MR3895752} and construct a stopping family $\mathcal{S}_\lambda(Q)$ of pairwise disjoint balls
whose $5$-dilations cover the set $Q^\lambda\subset Q^*=B(w,4r_Q)$; recall  \eqref{e.super}. \ed 
Let
$B\in\mathcal{B}_{Q}$ be a ball centered at  $x_B\in E_Q\subset \iol{Q}$. The {\em parent ball} of $B$ is then defined to be $\pi(B)=2B$ if $2B\subset Q^*$ and $\pi(B)=Q^*$ otherwise.
Observe that $B\subset \pi(B)\in\mathcal{B}_Q$ and  the center of $\pi(B)$ satisfies $x_{\pi(B)}
\in\{x_B,w\}\subset E_Q$.  It follows that all the balls
$B\subset \pi(B)\subset \pi(\pi(B))\subset \dotsb$  are well-defined, belong to $\mathcal{B}_Q$ and are centered at $E_Q$.
By inequalities \eqref{e.doubling} and \eqref{e.diams}, and property (W1) if needed,  we have $\mu(\pi(B))\le c_\mu^5 \mu(B)$
and $\diam(\pi(B))\le 16\diam(B)$.

Then we come to the stopping time argument. 
We will use as a threshold value the number
\[
\lambda_Q=\bigg(\frac{1}{\mathrm{diam}(Q^*)^{\beta p}} \vint_{Q^*} \lvert u(y)\rvert^p\,d\mu(y)\bigg)^{1/p}=\bigg(\frac{\mathbf{1}_{E_Q}(w)}{\mathrm{diam}(Q^*)^{\beta p}} \vint_{Q^*} \lvert u(y)\rvert^p\,d\mu(y)\bigg)^{1/p}.
\]
Fix a level $\lambda>\lambda_Q/2$. Fix a point $x\in Q^\lambda\subset Q^*$.
If $\lambda_Q/2<\lambda<\lambda_Q$, then we choose
$B_x=Q^*\in\mathcal{B}_{Q}$. If 
$\lambda\ge \lambda_Q$, then by using the condition $x\in Q^\lambda$ we first choose a starting ball $B$, with $x\in B\in\mathcal{B}_Q$, such that
\[
\lambda <
\bigg(\frac{\mathbf{1}_{E_Q}(x_B)}{\mathrm{diam}(B)^{\beta p}} \vint_{B} \lvert u(y)\rvert^p\,d\mu(y)\bigg)^{1/p}\,.
\]
Observe that $x_B\in E_Q\subset \iol{Q}$. 
We continue by looking at the balls $B\subset \pi(B) \subset \pi(\pi(B))\subset \dotsb$ 
and we 
stop at the first among them, denoted by $B_x\in\mathcal{B}_{Q}$, that satisfies the following  two  stopping conditions: 
\begin{align*}
\begin{cases} 
\lambda <
\displaystyle\bigg(\frac{\mathbf{1}_{E_Q}(x_{B_x})}{\mathrm{diam}(B_x)^{\beta p}} \vint_{B_x} \lvert u(y)\rvert^p\,d\mu(y)\bigg)^{1/p},\\
\displaystyle \bigg(\frac{\mathbf{1}_{E_Q}(x_{\pi(B_x)})}{\mathrm{diam}(\pi(B_x))^{\beta p}} \vint_{\pi(B_x)} \lvert u(y)\rvert^p\,d\mu(y)\bigg)^{1/p}\le \lambda \,.
 \end{cases}
\end{align*}
The inequality $\lambda\ge \lambda_Q$ in combination with 
 the fact that  $Q^*\subsetneq X$  ensures the existence of such a stopping ball.

In any case,  the chosen ball $B_x\in\mathcal{B}_Q$  contains
the point $x$, is centered at $x_{B_x}\in E_Q$, \ed and satisfies inequalities
\begin{equation}\label{e.loc_stop}
\lambda<
\bigg(\frac{\mathbf{1}_{E_Q}(x_{B_x})}{\mathrm{diam}(B_x)^{\beta p}} \vint_{B_x} \lvert u(y)\rvert^p\,d\mu(y)\bigg)^{1/p}\le 16c_\mu^{5/p} \lambda\ed\,.
\end{equation}
By the $5r$-covering lemma \cite[Lemma~1.7]{MR2867756},  we obtain a countable disjoint family 
\[\mathcal{S}_\lambda(Q)\subset \{B_x\,:\, x\in Q^\lambda\}\,,\qquad \lambda>\lambda_Q/2\,,\]
of {\em stopping balls}  such that $Q^\lambda\subset \bigcup_{B\in\mathcal{S}_\lambda(Q)}
5B$.
Let us remark that, by the condition (W2) and stopping inequality \eqref{e.loc_stop}, we have $B\subset U^{\lambda}$ 
if
$B\in \mathcal{S}_\lambda(Q)$ and $\lambda>\lambda_Q/2$.

\subsection{Level set estimates}\label{ss.auxiliary_local}

Next we prove
two technical results: Lemma \ref{l.dyadic} and Lemma \ref{l.mainl_local}.
We follow the approach in \cite{MR3895752} quite closely, but
we give  details since technical modifications are required.
A counterpart of the following lemma can be found also in \cite[Lemma 3.1.2]{MR2415381}.
Recall that $k\in\N$ is a fixed number and
$\alpha=\beta p^2/(2(s+\beta p))>0$ with $s=\log_2 c_\mu> 0$.

\begin{lemma}\label{l.dyadic}
Suppose that $\lambda>\lambda_Q/2$ and 
let $B\in\mathcal{S}_\lambda(Q)$
be such that $\mu (U_{B}^{2^k\lambda}) < \mu(B)/2$.
Then
\begin{equation}\label{e.abso}
\begin{split}
&\frac{1}{\mathrm{diam}(B)^{\beta p}}\int_{U_{B}^{2^k\lambda}}\lvert u(x)\rvert^p\,d\mu(x)\\&\qquad\le C(p,c_\mu)2^{-k\alpha} (2^{k}\lambda)^p 
\mu(U_{B}^{2^k\lambda})
+\frac{C(p,c_\mu)}{\mathrm{diam}(B)^{\beta p}}\int_{B\setminus U^{2^k\lambda}} \lvert u(x)\rvert^p\,d\mu(x)\,.
\end{split}
\end{equation}
\end{lemma}

\begin{proof}
%Fix $\lambda>\lambda_Q/2$ and  let $B\in\mathcal{S}_\lambda(Q)$ be such that
%$\mu (U_{B}^{2^k\lambda}) < \mu(B)/2$. 
Fix $x\in U_{B}^{2^k\lambda}\subset B$ and consider
the function $h\colon (0,\infty)\to\R$, 
\[
r\mapsto h(r)= \frac{\mu(U_{B}^{2^k\lambda}\cap B(x,r))}{\mu(B\cap B(x,r))}=\frac{\mu(U_{B}^{2^k\lambda}\cap B(x,r))}{\mu(B(x,r))}\cdot \bigg(\frac{\mu(B\cap B(x,r))}{\mu(B(x,r))}\bigg)^{-1}\,.
\]
By Lemma \ref{l.continuous} and the fact that $B$ is open, we find that $h\colon (0,\infty)\to \R$ is continuous.
 Observe that $U_{B}^{2^k\lambda}=U^{2^k\lambda}\cap B$ is also open.  Since $h(r)=1$ for small values of $r>0$
and $h(r)<1/2$ for  $r>\diam(B)$, we have $h(r_x)=1/2$ for some $0<r_x\le \diam(B)$.
Write $B'_x=B(x,r_x)$. Then
\begin{equation}\label{e.tok}
\frac{\mu(U_{B}^{2^k\lambda}\cap  B'_x)}{\mu(B\cap B'_x)}=h(r_x)=\frac{1}{2}
\end{equation}
and
\begin{equation}\label{e.ens}
\frac{\mu((B\setminus U^{2^k\lambda})\cap  B'_x)}{\mu(B\cap  B'_x)}
=1-\frac{\mu(U_{B}^{2^k\lambda}\cap  B'_x)}{\mu(B\cap B'_x)}
= 1-h(r_x)=\frac{1}{2}\,.
\end{equation}
The $5r$-covering lemma \cite[Lemma~1.7]{MR2867756}
gives us a  countable disjoint family  $\mathcal{G}_\lambda\subset \{ B'_x\,:\,  x\in U_{B}^{2^k\lambda}\}$
such that $U_{B}^{2^k\lambda}\subset \bigcup_{B'\in\mathcal{G}_\lambda} 
5B'$.
Then 
\eqref{e.tok} and \eqref{e.ens}  hold for every ball $B'\in
\mathcal{G}_\lambda$; namely,
by denoting $B'_I=U_{B}^{2^k\lambda}\cap B'$ and
${B'_O}=(B\setminus U^{2^k\lambda})\cap B'$,
we have the following comparison identities:
\begin{equation}\label{e.comparison}
\mu(B'_I)=  \frac{\mu( B\cap B')}{2}=  
\mu({B'_O})\,,
\end{equation}
where all the measures are strictly positive. These identities
are important and they  are used several times throughout the remainder of this proof.

We multiply the left-hand side of \eqref{e.abso}
by $\diam(B)^{\beta p}$ and then estimate as follows:
\begin{equation}\label{e.prepare}
\begin{split}
\int_{U_{B}^{2^k\lambda}}  &\lvert u\rvert^p\,d\mu
\le\sum_{B'\in\mathcal{G}_\lambda} \int_{5B'\cap B}\lvert u\rvert^p\,d\mu\\
&\le 2^{p-1}\sum_{B'\in\mathcal{G}_\lambda} \mu(5B'\cap B) \lvert u_{{B'_O}}\rvert^p+
2^{p-1}\sum_{B'\in\mathcal{G}_\lambda} \int_{5B'\cap B}\lvert u-u_{{B'_O}}\rvert^p\,d\mu\,.
\end{split}
\end{equation}
By  \eqref{e.doubling} and Lemma \ref{l.ball_measures},
we find that
\begin{equation}\label{e.pienenee}
\mu(5B'\cap B)\le \mu(8B') \le c_\mu^3\mu(B')\le c_\mu^6 \mu(B\cap B')
\end{equation}
 for all $B'\in\mathcal{G}_\lambda$.
Hence, by the comparison identities \eqref{e.comparison},
 \begin{equation}\label{e.eka}
\begin{split}
2^{p-1}&\sum_{B'\in\mathcal{G}_\lambda} \mu(5B'\cap B)  \lvert u_{{B'_O}}\rvert^p
\le C(p,c_\mu) \sum_{B'\in\mathcal{G}_\lambda} \mu({B'_O})  
\vint_{{B'_O}} \lvert u\rvert^p\,d\mu
\\&=C(p,c_\mu)\sum_{B'\in\mathcal{G}_\lambda}  
\int_{{B'_O}} \lvert u\rvert^p\,d\mu
\le C(p,c_\mu)
\int_{B\setminus U^{2^k\lambda}} \lvert u\rvert^p\,d\mu\,.
\end{split}
\end{equation}
This concludes our analysis of the `easy term' in
\eqref{e.prepare}.
In order to treat the remaining term therein, we do need some preparations.

Let us fix a ball $B'\in\mathcal{G}_\lambda$ that satisfies
$\int_{5B'\cap B} \lvert u-u_{{B'_O}}\rvert^p\,d\mu\not=0$. 
We claim that
\begin{equation}\label{e.out}
\vint_{5B'\cap B}\lvert u-u_{{B'_O}}\rvert^p\,d\mu\le  C(p,c_\mu) 2^{-k\alpha } (2^k\lambda)^p \diam(B)^{\beta p}\,.
\end{equation}
In order to prove this inequality, we fix a number $m\in \R$ such that 
\begin{equation}\label{e.m_def}
(2^m \lambda)^p \diam(5B')^{\beta p}=\vint_{5B'\cap B}\lvert u-u_{{B'_O}}\rvert^p\,d\mu\,.
\end{equation}
Let us first consider the case $m< k/2$. Then $m-k<-k/2$, and since always $\alpha<p/2$, the desired inequality \eqref{e.out} is obtained case as follows:
\begin{align*}
\vint_{5B'\cap B}\lvert u-u_{{B'_O}}\rvert^p\,d\mu
&=2^{(m-k)p} (2^k\lambda)^p\diam(5B')^{\beta p} \\&\le 10^p\, 2^{-kp/2}(2^k\lambda)^p\diam(B)^{\beta p}
\le C(p)  2^{-k\alpha }(2^k\lambda)^p\diam(B)^{\beta p}\,.
\end{align*}

Next we consider the case $k/2\le m$.
 Observe from \eqref{e.pienenee} and
%$\mu(5B')\le c_\mu^6\mu(5B'\cap B)$. By using also  \eqref{e.pienenee}  and
the comparison identities \eqref{e.comparison} that \ed
\begin{align*}
\vint_{5B'\cap B} \lvert u-u_{{B'_O}}\rvert^p\,d\mu &\le  
2^{p-1}\vint_{5B'\cap B} \lvert u-u_{5B'}\rvert^p\,d\mu + 2^{p-1}\lvert 
u_{5B'}-u_{{B'_O}}\rvert^p
\\&\le 2^{p+1}c_\mu^6 \vint_{5B'} \lvert u-u_{5B'}\rvert^p\,d\mu
\le 2^{p+1}c_\mu^6 (2^k\lambda)^p\mathrm{diam}(5B')^{\beta p}\,,
\end{align*}
where the last step follows from condition (W3) and the fact that $5B'\supset {B'_O}\not=\emptyset$.
By taking also \eqref{e.m_def} into account, we see that 
$2^{mp}\le 2^{p+1} c_\mu^6 2^{kp}$.
On the other hand, we have
\begin{align*}
(2^m \lambda)^p \diam(5B')^{\beta p}\mu(B'\cap B)&\le 
\int_{5B'\cap B} \lvert u-u_{{B'_O}}\rvert^p\,d\mu\\
&\le  2^{p-1}\int_{5B'\cap B} \lvert u\rvert^p\,d\mu 
+ 2^{p-1}\mu(5B'\cap B) \lvert u_{{B'_O}}\rvert^p\\
&\le 2^{p+1} c_\mu^6\int_B \lvert u\rvert^p\,d\mu\le 2 \cdot  32^p\ed c_\mu^{11}\ed \lambda^p \diam(B)^{\beta p}\mu(B)\,,
\end{align*}
where the last step follows from the fact that $B\in\mathcal{S}_\lambda(Q)$ in combination with
inequality \eqref{e.loc_stop}.
In particular, if $s=\log_2 c_\mu$ then by inequality \eqref{e.radius_measure}
and Lemma \ref{l.ball_measures}, we obtain that
\begin{align*}
\bigg(\frac{\diam(5B')}{\diam(B)}\bigg)^{s+\beta p} &\le 20^s \frac{\diam(5B')^{\beta p}\mu(B')}{\diam(B)^{\beta p}\mu(B)}
\le 20^s \cdot c_\mu^3\frac{\diam(5B')^{\beta p}\mu(B'\cap B)}{\diam(B)^{\beta p}\mu(B)}\\
&\le  2\cdot 20^s\cdot 32^p  \cdot c_\mu^{14} \cdot 2^{-mp}\le 2\cdot 20^s \cdot 32^p \cdot c_\mu^{14} \cdot 2^{-kp/2}\,.
\end{align*}
This, in turn, implies that
\[
\bigg(\frac{\diam(5B')}{\diam(B)}\bigg)^{\beta p} \le 2\cdot 20^s \cdot 32^p \cdot c_\mu^{14} \cdot 2^{\frac{-k\beta p^2}{2(s+\beta p)}}=  C(p,c_\mu) 2^{-k\alpha}\,.
\]
Combining the above estimates, we see that
\[
\vint_{5B'\cap B}\lvert u-u_{{B'_O}}\rvert^p\,d\mu=(2^m \lambda)^p \diam(5B')^{\beta p}\le C(p,c_\mu) 2^{-k\alpha } (2^k\lambda)^p \diam(B)^{\beta p}\,.
\]
That is, inequality \eqref{e.out} holds also in the present case $k/2\le m$.
This concludes the proof of inequality \eqref{e.out}.

By using  \eqref{e.pienenee} and \eqref{e.comparison} and inequality  \eqref{e.out}, we  estimate the second term in \eqref{e.prepare} as follows:
\begin{align*}
2^{p-1}\sum_{B'\in\mathcal{G}_\lambda} \int_{5B'\cap B}\lvert u-u_{{B'_O}}\rvert^p\,d\mu
&\le 2^p c_\mu^6\sum_{B'\in\mathcal{G}_\lambda} \mu(B'_I) \vint_{5B'\cap B}\lvert u-u_{{B'_O}}\rvert^p\,d\mu\\
&\le  C(p,c_\mu) 2^{-k\alpha } (2^k\lambda)^p \diam(B)^{\beta p}\sum_{B'\in\mathcal{G}_\lambda} \mu(B'_I)  
\\&\le C(p,c_\mu)2^{-k\alpha } (2^k\lambda)^p \diam(B)^{\beta p}\mu(U^{2^k\lambda}_B)\,. 
\end{align*}
Inequality \eqref{e.abso} follows by collecting the above estimates.
\end{proof}

The following lemma is essential for the proof of Theorem \ref{t.main_local}, and it is the only place
in the proof where
the capacity density condition is needed. Recall from Lemma~\ref{l.b_poincare} that
this condition implies a local boundary Poincar\'e  inequality, which is used here one single time.

\begin{lemma}\label{l.mainl_local}
Let $\lambda>\lambda_Q/2$ and  $g\in\mathcal{D}_H^{\beta}(u)$. Then
\begin{equation}\label{e.id}
\begin{split}
\lambda^p \mu(Q^\lambda)
\le C(\beta,p,c_\mu)\biggl[\frac{(\lambda 2^{k})^p}{2^{k\alpha}} \mu(U^{2^k \lambda})+   \frac{K}{k^p} \sum_{j=k}^{2k-1}  (\lambda 2^{j})^p \mu(U^{2^j \lambda})
+  K\int_{U^{\lambda}\setminus U^{4^k\lambda}} g^p\,d\mu\biggr]\,.
\end{split}
\end{equation}
% Here one can take 
%\[
%C_1'=2^p c_\mu^3C_1''\,,\qquad C_2'=2^p c_\mu^3(192 c_\mu^2)^p(1+C_2'')\,,\qquad 
%C_3'=8^p c_\mu^3 (1+C_2'')\,.
%\]
\end{lemma}

\begin{proof}
%Fix $\lambda>\lambda_Q/2$ and $g\in\mathcal{D}_H^{\beta}(u)$.
By the covering property $Q^\lambda\subset \bigcup_{B\in\mathcal{S}_\lambda(Q)}
5B$ and doubling condition \eqref{e.doubling},
\[\lambda^p \mu(Q^\lambda)
\le \lambda^p \sum_{B\in\mathcal{S}_\lambda(Q)}
\mu(5B) \le c_\mu^3 \sum_{B\in\mathcal{S}_\lambda(Q)} \lambda^p \mu(B)\,.\]
Recall also that $B\subset U^{\lambda}$ if $B\in\mathcal{S}_\lambda(Q)$.
Therefore, and using the fact that $\mathcal{S}_\lambda(Q)$ is a disjoint family, it suffices
to prove that inequality
\begin{equation}\label{e.local}
\begin{split}
\lambda^p \mu(B)
\le C(\beta,p,c_\mu)\biggl[\frac{(\lambda 2^{k})^p}{2^{k\alpha}} \mu(U^{2^k \lambda}_{B})+
\frac{K}{k^p} \sum_{j=k}^{2k-1}  (\lambda 2^{j})^p \mu(U_B^{2^j \lambda})
+ K\int_{B\setminus U^{4^k\lambda}} g^p\,d\mu\biggl]
\end{split}
\end{equation}
holds for every $B\in\mathcal{S}_\lambda(Q)$.
To this end, let us fix a ball $B\in\mathcal{S}_\lambda(Q)$.

If $\mu(U_B^{2^k\lambda})\ge \mu(B)/2$, then
\[
\lambda^p \mu(B) \le 2\lambda^p \mu(U_B^{2^k\lambda})
=2\frac{(\lambda 2^k)^p}{2^{kp}}\mu(U_B^{2^k\lambda})
\le 2\frac{(\lambda 2^k)^p}{2^{k \alpha}}\mu(U_B^{2^k\lambda})\,,
\]
which suffices for the required local estimate \eqref{e.local}.
Let us then consider the more difficult case $\mu(U_B^{2^k\lambda}) < \mu(B)/2$.
In this case, by the stopping inequality \eqref{e.loc_stop},
\begin{align*}
\lambda^p\mu(B)&\le \frac{\mathbf{1}_{E_Q}(x_B)}{\mathrm{diam}(B)^{\beta p}}\int_{B} \lvert u(x)\rvert^p\,d\mu(x)
\\&= \frac{\mathbf{1}_{E_Q}(x_B)}{\mathrm{diam}(B)^{\beta p}}\int_{X} \Bigl( \mathbf{1}_{B\setminus U^{2^k\lambda}}(x)+
\mathbf{1}_{U^{2^k\lambda}_{B}}(x)\Bigr)\lvert u(x)\rvert^p\,d\mu(x)\,.
\end{align*}
By Lemma
\ref{l.dyadic} it suffices to estimate 
the integral over the set
%:
$B\setminus U^{2^k\lambda}=B\setminus U^{2^k\lambda}_B$;
observe that the measure of this set is strictly positive.
We remark that the local boundary Poincar\'e
inequality in Lemma \ref{l.b_poincare} will be used to estimate this integral.

Fix a number $i\in\N$.
Since $B\subset Q^*$, it follows from Lemma \ref{l.arm_local} that the restriction $u|_{E_Q\cup (B\setminus U^{2^i \lambda})}$ is a 
$\beta$-H\"older function with a constant $\kappa_i=C(\beta,c_\mu)2^i\lambda$. 
We can now use the McShane extension \eqref{McShane} and extend 
$u|_{E_Q\cup (B\setminus U^{2^i \lambda})}$
to a function $u_{2^i \lambda}\colon X\to \R$
that is $\beta$-H\"older with the constant $\kappa_i$
and satisfies the restriction identity 
\[
u_{2^i \lambda}(x)=u(x)
\]
for all $x\in E_Q\cup (B\setminus U^{2^i \lambda})$.
Observe that $u_{2^i \lambda}=0$  in $E_Q$, since $u=0$  in $E_Q$. 

The crucial idea that was originally used by Keith--Zhong in \cite{MR2415381} is to consider the function
\[h(x)=\frac{1}{k} \sum_{i=k}^{2k-1} u_{2^i \lambda}(x)\,,\qquad x\in X\,.\]
We want to apply Lemma \ref{l.glueing}. In order to do so,  observe that $u_{2^i \lambda}|_{X\setminus A}=u|_{X\setminus A}$,
where \[A=X\setminus (B\setminus U^{2^i\lambda})=X\setminus (B\setminus U_B^{2^i\lambda})=(X\setminus B)\cup U^{2^i\lambda}_B\,.\]
Therefore, by Lemma \ref{l.glueing}
and properties (D1)--(D2), we obtain that 
\[
g_h=\frac{1}{k}\sum_{i=k}^{2k-1} \Bigl( \kappa_i  \mathbf{1}_{ (X\setminus B)\cup U^{2^i\lambda}_B\ed}
+ g\mathbf{1}_{B\setminus U^{2^i\lambda}}\Bigr)\in\mathcal{D}_H^{\beta}(h)\,.
\]
Observe that $U^{2^k\lambda}_B\supset U^{2^{(k+1)}\lambda}_B
\supset \dotsb \supset U^{2^{(2k-1)}\lambda}_B\supset U^{4^{k}\lambda}_B$.
By using these inclusions it is straightforward to show that the following
pointwise estimates are valid in $X$,
\begin{equation*}\label{e.hardy}
\begin{split}
\mathbf{1}_Bg_h^p &\le \bigg( \frac{1}{k}\sum_{i=k}^{2k-1} \Bigl(   \kappa_i\,\mathbf{1}_{U_B^{2^i\lambda}}
+  g \mathbf{1}_{B\setminus U^{2^i\lambda}}\Bigr)\bigg)^p\\
&\le 2^{p}\bigg(\frac{1}{k}\sum_{i=k}^{2k-1} \kappa_i\, \mathbf{1}_{U_B^{2^i \lambda}}\bigg)^p
+  2^{p} g^p \mathbf{1}_{B\setminus U^{4^k\lambda}}\\
&\le \frac{C(\beta,p,c_\mu)}{k^p} \sum_{j=k}^{2k-1} \bigg(\sum_{i=k}^j  2^i \lambda\bigg)^p  \mathbf{1}_{U_B^{2^j \lambda}}
+  2^p g^p \mathbf{1}_{B\setminus U^{4^k\lambda}}\\
&\le \frac{C(\beta,p,c_\mu)}{k^p} \sum_{j=k}^{2k-1}  (\lambda 2^{j})^p  \mathbf{1}_{U_B^{2^j \lambda}}
+  2^p g^p \mathbf{1}_{B\setminus U^{4^k\lambda}}\,.
\end{split}
\end{equation*}
Observe that $h\in \Lip_\beta(X)$ is zero in $E_Q$ and $h$ coincides with $u$ on $B\setminus U^{2^k\lambda}$, 
and recall that $g_h\in\mathcal{D}_H^{\beta}(h)$. Notice also that $B\subset Q^*$ and $x_B\in E_Q$.
The local boundary  Poincar\'e inequality in Lemma \ref{l.b_poincare} implies 
that
\begin{align*}
&\frac{\mathbf{1}_{E_Q}(x_B)}{\mathrm{diam}(B)^{\beta p}}\int_{B\setminus U^{2^k\lambda}}\lvert u(x)\rvert^p\,d\mu(x)
=\frac{\mathbf{1}_{E_Q}(x_B)}{\mathrm{diam}(B)^{\beta p}}\int_{B} \lvert h(x)\rvert^p\,d\mu(x)
\\&\le K\int_{B} g_h(x)^p \,d\mu(x)
\\&\le \frac{C(\beta,p,c_\mu) K}{k^p} \sum_{j=k}^{2k-1}  (\lambda 2^{j})^p \mu(U_B^{2^j \lambda})
+2^p K\int_{B\setminus U^{4^k\lambda}} g(x)^p\,d\mu(x)\,.
\end{align*}
The desired local inequality \eqref{e.local} follows by combining the estimates above.
\end{proof}

\subsection{Completing proof of Theorem \ref{t.main_local}}\label{ss.main_local}

We complete the proof as in \cite{MR3895752}.
Recall that  $u\colon X\to \R$ is a $\beta$-H\"older function with
$u=0$ in $E_Q$  and
%$U^\lambda=\{x\in B_0\,:\, M^{\sharp} u(x)>\lambda\}$.
 that \[
 M^{\sharp} u+M^{E_Q}u=M^{\sharp,p}_{\beta,\mathcal{B}_0}u+M^{E_Q,p}_{\beta,\mathcal{B}_0}u\,. 
 \]
Let us fix a function $g\in\mathcal{D}_H^{\beta}(u)$.
Observe that the left-hand side of 
inequality \eqref{e.loc_des} is finite. Without loss of generality, we may further assume
that it is nonzero.
%; in particular,
%from the latter condition it follows that $M^{\sharp,p}_\beta u(x)>0$ for
%all $x\in X$.
By Lemma \ref{l.big_t_o_small_ball},
\begin{align*}
&\int_{B_0} \big(M^{\sharp} u(x)+M^{E_Q}u(x)\big)^{p-\varepsilon}\,d\mu(x)
\\&\qquad \le   C(p,c_\mu)\ed\int_{B_0} \big(M^{E_Q}_{Q} u(x)\big)^{p-\varepsilon}\,d\mu(x) +  C(\beta,p,c_\mu)\ed \int_{B_0} g(x)^{p-\varepsilon}(x) \,d\mu(x)\,.
\end{align*}
We have
\begin{align*}
\int_{B_0} \big(M^{E_Q}_{Q} u(x)\big)^{p-\varepsilon}\,d\mu(x) =
\int_{Q^*} \big(M^{E_Q}_{Q} u(x)\big)^{p-\varepsilon}\,d\mu(x)= 
(p-\varepsilon)\int_0^\infty \lambda^{p-\varepsilon} \mu(Q^\lambda)\,\frac{d\lambda}{\lambda}\,.
\end{align*}
Since $Q^\lambda=Q^*=Q^{2\lambda}$ for every $\lambda \in (0,\lambda_Q/2)$, 
we find that
\begin{align*}
(p-\varepsilon)\int_0^{\lambda_Q/2} \lambda^{p-\varepsilon} \mu(Q^\lambda)\,\frac{d\lambda}{\lambda}&=\frac{(p-\varepsilon)}{2^{p-\varepsilon}}\int_0^{\lambda_Q/2} (2\lambda)^{p-\varepsilon} \mu(Q^{2\lambda})\,\frac{d\lambda}{\lambda}\\
&\le \frac{(p-\varepsilon)}{2^{p-\varepsilon}}\int_0^{\infty} \sigma^{p-\varepsilon} \mu(Q^{\sigma})\,\frac{d\sigma}{\sigma}
\\&=\frac{1}{2^{p-\varepsilon}}\int_{Q^*} \big(M^{E_Q}_{Q} u(x)\big)^{p-\varepsilon}\,d\mu(x)\,.
\end{align*}
On the other hand, by Lemma \ref{l.mainl_local}, for each $\lambda>\lambda_Q/2$,
\begin{align*}
\lambda^{p-\varepsilon} \mu(Q^\lambda)
\le C(\beta,p,c_\mu)\lambda^{-\varepsilon}\biggl[\frac{(\lambda 2^{k})^p}{2^{k\alpha}} \mu(U^{2^k \lambda})+ \frac{K}{k^p} \sum_{j=k}^{2k-1}  (\lambda 2^{j})^p \mu(U^{2^j \lambda})
+K\int_{U^{\lambda}\setminus U^{4^k\lambda}} g^p\,d\mu\,\biggr].
\end{align*}
Since $p-\varepsilon>1$, it follows that
\begin{align*}
\int_{Q^*} \big( M^{E_Q}_{Q} u(x)\big)^{p-\varepsilon}\,d\mu(x)
&\le 2(p-\varepsilon)\int_{\lambda_Q/2}^\infty \lambda^{p-\varepsilon} \mu(Q^\lambda)\,\frac{d\lambda}{\lambda}
\\&\le C(\beta,p,c_\mu)(I_1(Q) + I_2(Q) + I_3(Q))\,,
\end{align*}
where
\begin{align*}
I_1(Q)&=\frac{2^{k\varepsilon}}{2^{k\alpha}}\int_{0}^\infty (\lambda 2^{k})^{p-\varepsilon} \mu(U^{2^k \lambda})  \,\frac{d\lambda}{\lambda}\,,\qquad \\
I_2(Q)&= \frac{K}{k^p}\sum_{j=k}^{2k-1} 2^{j\varepsilon}\int_0^\infty (2^j \lambda)^{p-\varepsilon} \mu(U^{2^j \lambda})\,\frac{d\lambda}{\lambda}\,, \\
I_3(Q)&= K
\int_0^\infty \lambda^{-\varepsilon} \int_{U^{\lambda}\setminus U^{4^k\lambda}} g(x)^p\,d\mu(x)\,\frac{d\lambda}{\lambda}\,.
\end{align*}
We estimate these three terms separately. First,
\begin{align*}
I_1(Q)
&\le \frac{2^{k(\varepsilon-\alpha)}}{p-\varepsilon}\int_{B_0} \big(M^{\sharp} u(x)+M^{E_Q}u(x)\big)^{p-\varepsilon}\,d\mu(x)
\\&\le 2^{k(\varepsilon-\alpha)} \int_{B_0} \big(M^{\sharp} u(x)+M^{E_Q}u(x)\big)^{p-\varepsilon}\,d\mu(x)\,.
\end{align*}
Second, 
\begin{align*}
I_2(Q)
&\le \frac{K}{k^p}\sum_{j=k}^{2k-1}  2^{j\varepsilon}\int_0^\infty (2^j \lambda)^{p-\varepsilon} \mu(U^{2^j \lambda})\,\frac{d\lambda}{\lambda}\\
&\le \frac{K}{k^p(p-\varepsilon)}\bigg(\sum_{j=k}^{2k-1} 2^{j\varepsilon}\bigg)\int_{B_0} \big(M^{\sharp} u(x)+M^{E_Q}u(x)\big)^{p-\varepsilon}\,d\mu
\\
&\le  \frac{K4^{k\varepsilon}}{k^{p-1}}\int_{B_0} \big(M^{\sharp} u(x)+M^{E_Q}u(x)\big)^{p-\varepsilon}\,d\mu\,.
\end{align*}
Third, by Fubini's theorem,
\begin{align*}
I_3(Q)&\le K\int_{B_0\setminus \{M^{\sharp} u+M^{E_Q}u=0\}} \bigg(   \int_0^\infty \lambda^{-\varepsilon}  \mathbf{1}_{U^{\lambda}\setminus U^{4^k\lambda}}(x)  \frac{d\lambda}{\lambda}  \bigg)g(x)^p\,d\mu(x)\\
& \le  C(k,\varepsilon)  K\int_{B_0\setminus \{M^{\sharp} u+M^{E_Q}u=0\}}g(x)^p (M^{\sharp} u(x)+M^{E_Q}u(x))^{-\varepsilon}\,d\mu(x)\,.\end{align*}
Combining the estimates above, we arrive at the desired conclusion.
\qed

\section{Local Hardy inequalities}\label{s.local_hardy}

We apply Theorem \ref{t.main_boundedness} in order to obtain a local Hardy
inequality, see  \eqref{e.int} in Theorem \ref{t.wannebo_local}. This inequality is then shown  to be self-improving, see Theorem \ref{t.improved}, and 
  in this respect we follow the strategy in \cite{MR3673660}. 
However, we remark that the easier Wannebo approach \cite{MR1010807} for establishing local Hardy inequalities as in \cite{MR3673660}
is not available to us, due to absence of pointwise Leibniz and chain rules
in the setting of Haj{\l}asz gradients.

\begin{theorem}\label{t.wannebo_local}
Let $X$ be a geodesic space. Let
$1<p<\infty$ and $0<\beta\le 1$. Let $E\subset X$ be a closed
set which satisfies the $(\beta,p)$-capacity density condition with a constant $c_0$.
Let $B_0=B(w,R)$ be a ball with $w\in E$ and $R<\diam(E)$.
Let  $E_Q$ be the truncation
of $E$ to the Whitney-type ball $Q$ as in Section \ref{s.truncation}.
Then  there exists a constant $C=C(\beta,p,c_\mu,c_0)$ such that
\begin{equation}\label{e.int}
\int_{B(w,R)\setminus E_Q}\frac{\lvert u(x)\rvert^p}{d(x,E_Q)^{\beta p}}\,d\mu(x) 
\le C\int_{B(w,R)} g(x)^p\,d\mu(x)
\end{equation}
holds whenever $u\in\mathrm{Lip}_\beta(X)$ is such that $u=0$ in $E_Q$ and $g\in\mathcal{D}_H^{\beta}(u)$.
\end{theorem}

\begin{proof}
Let $u\in\mathrm{Lip}_\beta(X)$ be such that $u=0$ in $E_Q$ and let $g\in\mathcal{D}_H^{\beta}(u)$.
 Lemma \ref{l.arm_local} implies that
\[
\lvert u(x)\rvert \le 
C(\beta,c_\mu)d(x,E_Q)^\beta \left(M^{\sharp,p}_{\beta,\mathcal{B}_0}u(x)+M^{E_Q,p}_{\beta,\mathcal{B}_0}u(x)\right)
\]
for all $x\in Q^*$. Therefore
\[
\int_{Q^*\setminus E_Q} \frac{\lvert u(x)\rvert^p}{d(x,E_Q)^{\beta p}}\,d\mu(x) \le C(\beta,p,c_\mu)
\int_{B(w,R)} \left(M^{\sharp,p}_{\beta,\mathcal{B}_0}u(x)+M^{E_Q,p}_{\beta,\mathcal{B}_0}u(x)\right)^p\,d\mu(x)\,.
\]
By Theorem \ref{t.main_boundedness}, we obtain
\begin{equation}\label{e.loc_hard}
\int_{Q^*\setminus E_Q} \frac{\lvert u(x)\rvert^p}{d(x,E_Q)^{\beta p}}\,d\mu(x)
\le C(\beta,p,c_\mu,c_0)\int_{B(w,R)} g(x)^p\,d\mu(x)\,.
\end{equation}
 It remains to bound the integral over $B(w,R) \setminus Q^*$.  Since $E_Q\subset \iol{Q}$ and $Q^*=4Q$, we have $d(x,E_Q)\ge 3r_Q>R/64$
for all $x\in B(w,R)\setminus Q^*$. Thus, we obtain
\begin{align*}
&\int_{B(w,R)\setminus Q^*} \frac{\lvert u(x)\rvert^p}{d(x,E_Q)^{\beta p}}\,d\mu(x)
\le \frac{64^{\beta p}}{R^{\beta p}}\int_{B(w,R)}\lvert u(x)\rvert^p\,d\mu(x)\\&\le 
 \frac{3^p64^{\beta p}}{R^{\beta p}}\left(\int_{B(w,R)} \lvert u(x)-u_{B(w,R)}\rvert^p\,d\mu(x)
+
\mu(B(w,R))\lvert u_{B(w,R)}-u_{Q^*}\rvert^p + \mu(B(w,R))\lvert u_{Q^*}\rvert^p\right)\,.
\end{align*}
By the $(\beta,p,p)$-Poincar\'e inequality in Lemma \ref{t.pp_poincare}, 
\begin{align*}
&\int_{B(w,R)} \lvert u(x)-u_{B(w,R)}\rvert^p\,d\mu(x)\\&\qquad \le 2^p\diam(B(w,R))^{\beta p}\int_{B(w,R)} g(x)^p\,d\mu(x)
\le C(p) R^{\beta p}\int_{B(w,R)} g(x)^p\,d\mu(x)\,.
\end{align*}
For the second term, we have
\begin{align*}
&\mu(B(w,R))\lvert u_{B(w,R)}-u_{Q^*}\rvert^p \\&\qquad \le \mu(B(w,R))\vint_{Q^*} \lvert u(x)-u_{B(w,R)}\rvert^p\,d\mu(x)
\\&\qquad \le C(c_\mu)\int_{B(w,R)} \lvert u(x)-u_{B(w,R)}\rvert^p\,d\mu(x)\le C(p,c_\mu)R^{\beta p}\int_{B(w,R)} g(x)^p\,d\mu(x)\,.
\end{align*}
For the third term, we have $d(x,E_Q)\le d(x,w)< 4r_Q<R$ for every $x\in Q^*$. Thus
\begin{align*}
\mu(B(w,R))\lvert u_{Q^*}\rvert^p\le C(c_\mu)\int_{Q^*\setminus E_Q} \lvert u(x)\rvert^p\,d\mu(x)
\le R^{\beta p}\int_{Q^*\setminus E_Q} \frac{\lvert u(x)\rvert^p}{d(x,E_Q)^{\beta p}}\,d\mu(x)\,.
\end{align*}
Applying inequality \eqref{e.loc_hard}, we get
\[
\mu(B(w,R))\lvert u_{Q^*}\rvert^p\le  C(\beta,p,c_\mu,c_0)R^{\beta p}\int_{B(w,R)} g(x)^p\,d\mu(x)\,.
\]
The desired inequality follows by combining the estimates above.
\end{proof}

Next we improve the local Hardy inequality 
in Theorem~\ref{t.wannebo_local}. This is done
by adapting the Koskela--Zhong  truncation argument from \cite{KoskelaZhong2003} to the setting
of Haj{\l}asz gradients; see also
\cite{MR3673660} and \cite[Theorem 7.32]{KLV2021} whose proof we modify to our purposes. 
 
\begin{theorem}\label{t.improved}
Let $X$ be a geodesic space. Let
$1<p<\infty$ and $0<\beta\le 1$. Let $E\subset X$ be a closed
set which satisfies the $(\beta,p)$-capacity density condition with a constant $c_0$.
Let $B_0=B(w,R)$ be a ball with $w\in E$ and $R<\diam(E)$.
Let  $E_Q$ be the truncation
of $E$ to the Whitney-type ball $Q$ as in Section \ref{s.truncation},
and let $C_1=C_1(\beta,p,c_\mu,c_0)$ be the constant in~\eqref{e.int}, see Theorem~{\textup{\ref{t.wannebo_local}}}.
Then there exist $0< \varepsilon=\varepsilon(p,C_1)<p-1$
and $C= C(p,C_1)$ such that inequality
\begin{equation}\label{e.himproved}
\int_{B(w,R)\setminus E_Q} \frac{\lvert u(x)\rvert^{p-\varepsilon}}{d(x,E_Q)^{\beta(p-\varepsilon)}}\,d\mu(x)
\le C \int_{B(w,R)} g(x)^{p-\varepsilon}\,d\mu(x)
\end{equation}
holds whenever $u\in\mathrm{Lip}_\beta(X)$ is such that $u=0$ in $E_Q$ and $g\in\mathcal{D}_H^{\beta}(u)$.
\end{theorem}

\begin{proof}
Without loss of generality, we may assume that $C_1\ge 1$ in \eqref{e.int}.
Let $u\in\mathrm{Lip}_\beta(X)$ be such that $u=0$ in $E_Q$ and let $g\in\mathcal{D}_H^{\beta}(u)$.
 Let $\kappa\ge 0$ be the $\beta$-H\"older constant of $u$ in $X$.
By redefining $g=\kappa$ in the exceptional set  $N=N(g)$ of measure zero, we may assume that
\eqref{e.hajlasz} holds for all $x,y\in X$. 
Let $\lambda>0$ and define $F_\lambda = G_\lambda \cap H_\lambda$, where 
\[
G_\lambda = \bigl\{x\in B(w,R) : g(x)\le \lambda\bigr\}
\]
and
\[
H_\lambda = \{x\in B(w,R) : \lvert u(x)\rvert \le \lambda d(x,E_Q)^\beta\}.
\]
We show that the restriction of $u$ to $F_\lambda\cup E_Q$ is
a $\beta$-H\"older function with a constant  $2\lambda$. Assume
that $x,y\in F_\lambda$. Then \eqref{e.hajlasz}  implies
\[
\lvert u(x)-u(y)\rvert\le d(x,y)^\beta\left(g(x)+g(y)\right)\le 2\lambda d(x,y)^\beta\,.
\]
On the other hand, if $x\in F_\lambda$ and $y\in E_Q$, then
\[
\lvert u(x)-u(y)\rvert =\lvert u(x)\rvert \le \lambda d(x,E_Q)^\beta\le  2\lambda d(x,y)^\beta\,.
\]
The case $x\in E_Q$ and $y\in F_\lambda$ is treated in the same way. If $x,y\in E_Q$, then $\lvert u(x)-u(y)\rvert=0$.
All in all, we see that $u$ is a $\beta$-H\"older function in $F_\lambda\cup E_Q$ with a constant $2\lambda$.

We apply the McShane extension~\ref{McShane} and extend the restriction $u\vert_{F_\lambda\cup E_Q}$
to a $\beta$-H\"older function function $v$ in $X$  with constant $2\lambda$.
Then $v=u=0$ in $E_Q$ and  $v=u$ in $F_\lambda$, thus
\[
g_v=  g\mathbf{1}_{F_\lambda}+2\lambda \mathbf{1}_{X\setminus F_\lambda} \in \mathcal{D}^\beta_H(v)
\]
by Lemma \ref{l.glueing}.

By applying Theorem~\ref{t.wannebo_local} to the
function $v$ and its Haj{\l}asz $\beta$-gradient $g_v$, we obtain 
\begin{align*}
\int_{(B(w,R)\setminus E_Q)\cap F_\lambda} \frac{\lvert u(x)\rvert^p}{d(x,E_Q)^{\beta p}}\,d\mu(x)
& \le \int_{B(w,R)\setminus E_Q} \frac{\lvert v(x)\rvert^p}{d(x,E_Q)^{\beta p}}\,d\mu(x)\\
&\le C_1\int_{F_\lambda} g(x)^p\,d\mu(x) +C_1 2^p\lambda^p \mu(B(w,R)\setminus F_\lambda)\,.
\end{align*}
Since $H_\lambda=F_\lambda\cup (H_\lambda\setminus G_\lambda)$ and $C_1\ge 1$, it follows that
\begin{equation}\label{e.nuva}
\begin{split}
&\int_{(B(w,R)\setminus E_Q)\cap H_\lambda} \frac{\lvert u(x)\rvert^p}{d(x,E_Q)^{\beta p}}\,d\mu(x)\\
& \qquad \le C_1 \int_{F_\lambda} g(x)^p\,d\mu(x)
 +C_1 2^p\lambda^p \mu( B(w,R)\setminus F_\lambda)\\
& \qquad\qquad+\int_{(H_\lambda\setminus E_Q)\setminus G_\lambda}\frac{\lvert u(x)\rvert^p}{d(x,E_Q)^{\beta p}}\,d\mu(x)\\
& \qquad \le C_1 \int_{G_\lambda} g(x)^p\,d\mu(x)
+C_1 2^{p} \lambda^p \bigl(\mu(B(w,R)\setminus F_\lambda)
+\mu( H_\lambda\setminus G_\lambda)\bigr)\\
& \qquad \le C_1 \int_{G_\lambda} g(x)^p\,d\mu(x)
+C_1 2^{p+1} \lambda^p \bigl(\mu( B(w,R)\setminus H_\lambda)+\mu( B(w,R)\setminus G_\lambda)\bigr).
\end{split}
\end{equation}
Here $\lambda>0$ was arbitrary, and thus we conclude that~\eqref{e.nuva} holds for every $\lambda>0$.

Next we multiply~\eqref{e.nuva} by $\lambda^{-1-\varepsilon}$,
where $0<\varepsilon<p-1$, and integrate with respect to $\lambda$ over the set $(0,\infty)$. 
With a change of the order of integration on the left-hand side, this gives
\begin{align*}
& \frac{1}{\eps} \int_{B(w,R)\setminus E_Q} 
\biggl(\frac{\lvert u(x)\rvert}{d(x,E_Q)^\beta}\biggr)^{p-\eps}\,d\mu(x)
 \le C_1 \int_0^\infty \lambda^{-1-\varepsilon}\int_{G_\lambda} g(x)^p\,d\mu(x)\,d\lambda \\
& \qquad + C_1 2^{p+1}\int_0^\infty \lambda^{p-1-\varepsilon} \bigl(\mu(B(w,R)\setminus H_\lambda)
+\mu(B(w,R)\setminus G_\lambda)\bigr)\,d\lambda\,.
\end{align*}
By the definition of $G_\lambda$, we find that the first term
on the right-hand side is dominated by
\[
\frac{C_1}{\eps}\int_{B(w,R)} g(x)^{p-\varepsilon}\,d\mu(x)\,.
\]
Using the definitions
of $H_\lambda$ and $G_\lambda$, the second term on the right-hand side 
can be estimated from above by
\[
\frac{C_1 2^{p+1}}{p-\varepsilon}\biggl(
\int_{B(w,R)\setminus E_Q} \biggl(\frac{\lvert u(x)\rvert}{d(x,E_Q)^\beta}\biggr)^{p-\eps}\,d\mu(x)
+ \int_{B(w,R)} g(x)^{p-\varepsilon}\,d\mu(x)\biggr).
\]

By combining the estimates above, we obtain
\begin{equation}\label{e.endial}
\begin{split}
&\int_{B(w,R) \setminus E_Q} \biggl(\frac{\lvert u(x)\rvert}{d(x,E_Q)^\beta}\biggr)^{p-\eps}\,d\mu(x)\\
&\qquad\le C_2\int_{B(w,R)\setminus E_Q} \biggl(\frac{\lvert u(x)\rvert}{d(x,E_Q)^\beta}\biggr)^{p-\eps}\,d\mu(x) + 
   C_3\int_{B(w,R)} g(x)^{p-\varepsilon}\,d\mu(x)\,,
   \end{split}
\end{equation}
where $C_2 = C_1 2^{p+1} \frac{\varepsilon} {p-\varepsilon}$ 
and $C_3 = C_1\bigl(1+2^{p+1}\frac{\varepsilon} {p-\varepsilon}\bigr)$. We choose
$0<\varepsilon=\varepsilon(C_1,p)<p-1$ so small that
\[
C_2 = C_1 2^{p+1} \frac{\varepsilon} {p-\varepsilon}<\frac{1}{2}\,.
\]
This allows us to absorb the first term in the right-hand side of \eqref{e.endial} to the left-hand side.
Observe that this term is finite, since $u$ is $\beta$-H\"older in $X$ and $u=0$ in $E_Q$.
\end{proof}

\section{Self-improvement of the capacity density condition}\label{s.big-results}

 As an application of Theorem \ref{t.improved}, we  strengthen Theorem \ref{t.easy_converse} in  complete geodesic spaces.
This leads to the conclusion that the Haj{\l}asz capacity density condition 
is self-improving or  doubly open-ended in such spaces.
In fact, we characterize the Haj{\l}asz capacity density condition in various geometrical and 
analytical quantities, the latter of which are all shown to be doubly open-ended.

\begin{theorem}\label{t.converse}
Let $X$ be a geodesic space. Let
$1<p<\infty$ and $0<\beta\le 1$. Let $E\subset X$ be a closed 
set which satisfies the $(\beta,p)$-capacity density condition with a constant $c_0$.
 Then
there exists $\varepsilon>0$, depending
on $\beta$, $p$, $c_\mu$ and $c_0$, such that
$\ucodima(E)\le \beta(p-\varepsilon)$.
\end{theorem}

\begin{proof}
Let $w\in E$ and $0<r<R<\diam(E)$. 
Let  $E_Q$ be the truncation
of $E$ to the  ball $Q\subset B_0=B(w,R)$ as in Section \ref{s.truncation}.
Let $\varepsilon>0$ be as in Theorem \ref{t.improved}.
Observe that
\[E_{Q,r}=\{x\in X\,:\,d(x,E_Q)<r\}\subset \{x\in X\,:\, d(x,E)<r\}=E_r\,.\]
Hence, it suffices to show that
\begin{equation}\label{eq.codim_est}
\frac{\mu(E_{Q,r}\cap B(w,R))}{\mu(B(w,R))}\ge c\Bigl(\frac{r}{R}\Bigr)^{\beta(p-\varepsilon)},
\end{equation}
where the constant $c$ is independent of $w$, $r$ and $R$.

If $r\ge R/4$, then the claim is clear since $\bigl(\frac{r}{R}\bigr)^{\beta(p-\varepsilon)}\le 1$ and
\[
\mu(E_{Q,r}\cap B(w,R))\ge \mu(B(w,R/4))\ge C(\mu)\mu(B(w,R))\,.
\]
The claim is clear also if  $\mu(E_{Q,r}\cap B(w,R))\ge \frac 1 2 \mu(B(w,R))$.
Thus we may assume that $r<R/4$ and that
$\mu(E_{Q,r}\cap B(w,R)) < \tfrac 1 2 \mu(B(w,R))$, whence
\begin{equation}\label{eq.compl_meas}
\mu(B(w,R)\setminus E_{Q,r}) \ge \tfrac 1 2 \mu(B(w,R))>0.
\end{equation}

Let us now consider the  $\beta$-H\"older function $u\colon X\to\R$,
\[
u(x)=\min\{1,r^{-\beta}d(x,E_Q)^\beta\}\,,\qquad  x\in X.
\]
Then $u=0$ in $E_Q$, $u=1$ in $X\setminus E_{Q,r}$, and 
\[
\lvert u(x)-u(y)\rvert \le r^{-\beta}d(x,y)^\beta \quad \text{ for all } x,y\in X.
\]

 We aim to apply Theorem \ref{t.improved}.  Recall also that $w\in E_Q$. Thus we obtain
\begin{equation}\label{eq.lhs_lower}
\begin{split}
\int_{B_0\setminus E_Q} \frac{\lvert u(x)\rvert^{p-\varepsilon}}{d(x,E_Q)^{\beta(p-\varepsilon)}}\,d\mu(x)
&\ge 
R^{-\beta(p-\varepsilon)}\int_{B_0\setminus E_Q} \lvert u(x)\rvert^{p-\varepsilon}\,d\mu(x)\\
&\ge R^{-\beta(p-\varepsilon)}\int_{B_0\setminus E_{Q,r}} \lvert u(x)\rvert^{p-\varepsilon}\,d\mu(x)
\\&\ge R^{-\beta(p-\varepsilon)}\mu(B(w,R)\setminus E_{Q,r})
\\& \ge 2^{-1}R^{-\beta(p-\varepsilon)}\mu(B(w,R))\,,
\end{split}
\end{equation}
where the last step follows from~\eqref{eq.compl_meas}.

Since $u=1$ in $X\setminus E_{Q,r}$ and
$u$ is a $\beta$-H\"older function with a constant $r^{-\beta}$, 
Lemma \ref{l.glueing} implies that
$g=r^{-\beta}\mathbf{1}_{E_{Q,r}}\in\mathcal{D}_H^{\beta}(u)$.
Observe that
\[
\int_{B_0} g^{p-\varepsilon}\,d\mu\le r^{-\beta(p-\varepsilon)}\mu(E_{Q,r}\cap B_0)=
r^{-\beta(p-\varepsilon)}\mu(E_{Q,r}\cap B(w,R))\,.
\]
Hence, the claim~\eqref{eq.codim_est} follows from \eqref{eq.lhs_lower} and
Theorem \ref{t.improved}. 
\end{proof}

The following  theorem is a compilation of the results in this paper. It states the equivalence of some geometrical conditions (1)--(2) and analytical conditions (3)--(6), one of which is the capacity density condition.
We emphasize that the capacity density condition (3)
is characterized in terms of the upper Assouad codimension (1); in fact, this characterization follows
immediately from Theorem~\ref{t.assouad_sufficient} and Theorem \ref{t.converse}.

\begin{theorem}\label{t.main_characterization}
Let $X$ be a complete geodesic space.  Let
$1<p<\infty$ and $0<\beta\le 1$. Let $E\subset X$ be a closed 
set. Then the following conditions are equivalent:
\begin{itemize}
\item[(1)] $\ucodima(E)<\beta p$.
\item[(2)] $E$ satisfies the Hausdorff content
density condition \eqref{e.hausdorff_content_density} for some $0<q<\beta p$.
\item[(3)] $E$ satisfies the $(\beta,p)$-capacity density condition.
\item[(4)] $E$ satisfies the local $(\beta,p,p)$-boundary Poincar\'e inequality \eqref{e.local_boundary_poincare}.
\item[(5)] $E$ satisfies the maximal $(\beta,p,p)$-boundary Poincar\'e inequality
\eqref{e.loc_des_C}.
\item[(6)] $E$ satisfies the local $(\beta,p,p)$-Hardy inequality \eqref{e.int}.
\end{itemize}
\end{theorem}

\begin{proof} 
The  implication from (1) to (2) is a consequence of Lemma \ref{l.measure_to_hausdorff} with $\ucodima(E)<q<\beta p$. The implication from (2) to (3) 
follows by adapting the proof of Theorem~\ref{t.assouad_sufficient} with $\eta=q/\beta$.
%The implication from (3) to (1) follows from Theorem \ref{t.converse}.
%The reverse implication follows from Theorem \ref{t.assouad_sufficient}.
The implication from (3) to (4) follows from Theorem \ref{l.b_poincare}. The implication
from (4) to (5) follows from the proof of Theorem \ref{t.main_boundedness}, which remains valid if we assume (4) instead of the $(\beta,p)$-capacity density condition.
The implication from (5) to (6) follows from the proof
of Theorem \ref{t.wannebo_local}.
Finally, condition (6) implies the improved
local Hardy inequality \eqref{e.himproved} and 
the proof of Theorem \ref{t.converse} then shows the remaining implication from (6) to (1).
\end{proof}

Finally, we state the  main result of this paper, Theorem~\ref{t.main-improvement}.  It is the self-improvement or  double open-endedness 
property of the $(\beta,p)$-capacity density condition. Namely, in addition to 
integrability exponent $p$, also the order $\beta$ of fractional differentiability can be lowered.
A similar phenomenon is observed 
in \cite{MR946438} for Riesz capacities in $\R^n$.
%See also  \cite{MR3336922}, where
%solutions to nonlocal equations with measurable coefficients 
%are shown to be both higher integrable and higher differentiable.

\begin{theorem}\label{t.main-improvement}
Let $X$ be a  complete geodesic space, and let
$1<p<\infty$ and $0<\beta\le 1$. Assume that
a closed  set $E\subset X$ satisfies the $(\beta,p)$-capacity density condition.
Then there exists $0<\delta<\min\{\beta,p-1\}$ such that 
$E$ satisfies the $(\gamma,q)$-capacity density condition for all
$\beta-\delta<\gamma\le 1$ and $p-\delta<q<\infty$.
\end{theorem}

\begin{proof}
We have $\ucodima(E)< \beta p$ by Theorem \ref{t.main_characterization}.
Since $\lim_{\delta\to 0}(\beta-\delta)(p-\delta)=\beta p$,
there exists
$0<\delta<\min\{\beta,p-1\}$ such that
$\ucodima(E)<(\beta-\delta)(p-\delta)$. 
Now if $\beta-\delta<\gamma\le 1$ and $p-\delta<q<\infty$, then
\[
\ucodima(E)<(\beta-\delta)(p-\delta)<\gamma q\,.\]
The claim
follows from Theorem \ref{t.main_characterization}.
\end{proof}

A similar argument shows that the analytical conditions (4)--(6) in Theorem \ref{t.main_characterization} are also doubly open ended.  The geometrical conditions (1)--(2) are open-ended by definition.

\def\cprime{$'$} \def\cprime{$'$} \def\cprime{$'$}

\end{document}